\documentclass[10pt]{article}

\usepackage[top=1.25in, bottom=1.25in, left=1.25in, right=1.25in]{geometry}
\newsavebox{\TableThree}
\usepackage{amsmath,amsthm,amssymb,xurl}
\usepackage{mathtools}
\usepackage{enumitem} 
\usepackage{xspace} 
\usepackage{xifthen}
\usepackage[hidelinks]{hyperref}
\usepackage{tikz,authblk}
\usepackage{textcomp} 
\usepackage{fancybox}
\usepackage{float}
\usepackage{longtable}
\usepackage{authblk}
\usepackage{xcolor}
\usepackage{soul} 
\usepackage{multicol}
 \usepackage{array}
 \usepackage{ragged2e}
 \usepackage{caption}
 \newcolumntype{R}[1]{>{\RaggedRight}p{#1}}
\usepackage[normalem]{ulem}

\usepackage{xparse}

\DeclareDocumentCommand{\NF}{o m m}%
{%
\IfValueTF{#1}
 {\ensuremath{#1\textsc{-nf}(#2,#3)}\xspace}
 {\ensuremath{\textsc{nf}(#2,#3)}\xspace}%
}

\newcounter{algcounter}
\renewcommand{\thealgcounter}{\arabic{algcounter}}

{\refstepcounter{algcounter}%
\label{#1}%
\begin{Sbox}\begin{minipage}{\boxwidth}%
\mbox{\quad}\vspace{-.025in}\\ \af{Algorithm~\thealgcounter : #2} 
\noindent}%
 {\mbox{\quad}\vspace{-.2in} \end{minipage}
\end{Sbox}\bigskip\noindent\fbox{\TheSbox}
}

\usetikzlibrary{shapes.geometric}
\newcommand{\que}{{\ensuremath{\mathbb{Q}}}}
\newcommand{\zed}{{\ensuremath{\mathbb{Z}}}}
\newcommand{\eff}{{\ensuremath{\mathbb{F}}}}

\newcommand{\Aut}{{\ensuremath{\mathsf{Aut}}}}

\newcommand{\af}[1]{\textbf{#1}} 

\newlength{\boxwidth}
\theoremstyle{plain}
\newtheorem{theorem}{Theorem}[section]

\newtheorem{lemma}[theorem]{Lemma}
\newtheorem{corollary}[theorem]{Corollary}
\newtheorem{remark}[theorem]{Remark}
\newtheorem{example}[theorem]{Example}

\floatstyle{plain}
\newfloat{flt}{htbp}{.flt}[section]

\allowdisplaybreaks

\title{$\lambda$-fold near-factorizations of groups
}
\author[1]{Donald L.\ Kreher}
\affil[1]{Department of Mathematical Sciences,
Michigan Technological University,
Houghton, MI 49931, U.S.A.}
\author[2]{Shuxing Li}
\affil[2]{Department of Mathematical Sciences, University of Delaware, Newark, DE 19716, U.S.A.}
\author[3]{Douglas R.\ Stinson\thanks{D.R.\ Stinson's research is supported by  NSERC discovery grant RGPIN-03882.}}
\affil[3]{David R.\ Cheriton School of Computer Science, University of Waterloo, Waterloo ON, N2L 3G1, Canada}

\begin{document}
\maketitle

\begin{abstract}
We initiate the study of $\lambda$-fold near-factorizations of groups with $\lambda > 1$. While $\lambda$-fold near-factorizations of groups with $\lambda =  1$ have been studied in numerous papers, this is the first detailed treatment for 
$\lambda > 1$. We establish fundamental properties of $\lambda$-fold near-factorizations and introduce the notion of equivalence. We prove various necessary conditions of $\lambda$-fold near-factorizations, including upper bounds on $\lambda$. We present three constructions of infinite families of $\lambda$-fold near-factorizations, highlighting the characterization of two subfamilies of $\lambda$-fold near-factorizations. We discuss a computational approach to $\lambda$-fold near-factorizations and tabulate computational results for abelian groups of small order.
\end{abstract}

\section{$\lambda$-mates in near-factorizations}
\label{sec1}

We begin with some definitions involving the group ring $\zed[G]$ (for a gentle introduction to group rings and character theory, see Jedwab and Li \cite{JL}). For the purposes of these definitions, we assume that $G$ is a multiplicative group.

Suppose $(G,\cdot)$ is a group with identity $e$. 
Elements of $\zed[G]$ are formal sums
\[ \sum_{g \in G} a_g \, g,\] where the $a_g$'s are in $\zed$.
A subset $S \subseteq G$ corresponds in an obvious way to 
the group ring element
$ \sum_{g \in S}  g$. By abuse of notation, we also use $S \in \zed[G]$ to represent the group ring element corresponding to the subset $S$ of $G$.
Note that all coefficients of this group ring element are $0$'s and $1$'s. In particular, $G$
is identified with $ \sum_{g \in G}  g \in \zed [G]$ and $G \setminus \{e\}$ is identified with $G - e\in \zed [G]$. 

Operations in $\zed[G]$ are defined as follows. Suppose $A,B \in \zed [G]$, say 
\[A = \sum_{g \in G} a_g \, g \quad \text{and} \quad 
 B = \sum_{g \in G} b_g \, g.\] The sum $A+B$ is computed as $C = \sum_{g \in G} c_g \, g$, where
\[ c_g = a_g + b_g\] for all $g\in G$. The difference $A-B$ is computed in the obvious way.
If $A$ and $B$ have non-negative integer coefficients, then they can be viewed as multisets. In this case, $A + B$ denotes the multiset union of $A$ and $B$.

The product $AB$ is computed as 
$D =\sum_{g \in G} d_g \, g$, where
\[ d_g = \sum_{\{ (g',g'') \, : \, g' g'' = g\}} a_{g'} b_{g''} 
=\sum_{h \in G} a_{gh^{-1}}b_h
\] for all $g\in G$.
For an integer $c$, we define 
\[cA = \sum_{g \in G} c a_g \, g.\]
Finally, define 
\[ A^{(-1)} = \sum_{g \in G} a_g \, g^{-1}.\]

We can now define near-factorizations.
Let $(G,\cdot)$ be a finite multiplicative group with identity $e$. 
Suppose that $S, T \subseteq G$ (that is, $S$ and $T$ are required to be subsets of $G$). We say that $(S,T)$ is a \emph{$\lambda$-fold near-factorization of $G$}
(denoted \NF{G}{\lambda})
if the following group ring equation holds:
\[
ST= \lambda(G - e).
\]
This means that every non-identity group element $g$ can be expressed as a product $g= st$ 
with $s \in S$ and $t \in T$ in exactly $\lambda$ ways. Further, $st \neq e$ for any $s \in S$ and $t \in T$.

An \NF{G}{\lambda}, say  $(S,T)$, 
is an \emph{$(s,t)$-$\lambda$-fold near-factorization of $G$} (denoted \NF[(s,t)]{G}{\lambda})
if $|S|=s$ and $|T| =t$.
Note that an \NF[(s,t)]{G}{\lambda} exists only if 
\begin{equation}
\label{prod.eq}
st = \lambda(|G| -1).
\end{equation}
Observe that necessarily
$\lambda \leq s $ and $\lambda \leq t$. (Stronger upper bounds on $\lambda$ will be proven in Section \ref{lambda.sec}.)

There is always a \emph{trivial} \NF[(1,|G|-1)]{G}{1},
given by $S = \{e\}$, $T = 
G \setminus \{e\}$. More generally, if we take $S= \{ g \}$, $T=G \setminus \{ g^{-1} \}$ for any $g \in G$, then
we get a {trivial} \NF[(1,|G|-1)]{G}{1}.
A near-factorization with $|S| > 1$ and $|T| > 1$ is \emph{nontrivial}.
If $(S,T)$ is a $\lambda$-fold near-factorization, then we say that $T$ is a \emph{$\lambda$-mate} of $S$.
\textup{(A $\lambda$-mate with $\lambda=1$ is simply called a
\emph{mate}.)}

Virtually all of our constructions and 
examples of near-factorizations will be in additive groups. In an additive group, the product $D = AB$ of 
$A,B \in \zed[G]$ would be defined in terms of the additive (group) operation. Hence, we define 
$D =\sum_{g \in G} d_g \, g$, where
\[ d_g = \sum_{\{ (g',g'') \, : \, g' +  g'' = g\}} a_{g'} b_{g''} = \sum_{ h \in G } a_{g-h} b_{h}\] for all $g\in G$.
Note that we still denote this as a product in the group ring $\zed[G]$, even though the group operation is written additively.

\begin{example}
{\rm We present a \NF[(4,7)]{\zed_{15}}{2} consisting of $S = \{1,  4,  11,  14 \}$ and   $T= \{ 0,  2,  6,  7,  8,  9,  13\}$.
The group ring product $ST$ is computed as the multiset consisting of all sums $s \in S$, $t \in T$. We tabulate these sums as follows:
\[
\begin{array}{r||r|r|r|r|r|r|r}
 & 0 & 2 & 6 & 7 & 8 & 9 & 13 \\ \hline\hline
1 & 1 & 3 & 7 & 8 & 9 & 10 & 14 \\ \hline
4 & 4 & 6 & 10 & 11 & 12 & 13 & 2\\ \hline
11 & 11 & 13 & 2 & 3 & 4 & 5 & 9 \\ \hline
14 & 14 & 1 & 5 & 6 & 7 & 8 & 12
\end{array}
\]
We see that there are two occurrences of every nonzero element of $\zed_{15}$ in this table. That is, the group ring equation
$ST = 2(\zed_{15} - 0)$ holds and therefore we have a \NF[(4,7)]{\zed_{15}}{2}.
}$\hfill\blacksquare$
\end{example}

There has been considerable study of \NF[(s,t)]{G}{1}; see, for example, \cite{BHS,DB,CGHK,KPS,KPS2,KMS,Nathanson,Pech03,Pech,SS}. The 1956 paper by de Bruijn \cite{DB} was apparently the first one to study near-factorizations.
We believe that \NF[(s,t)]{G}{\lambda} with $\lambda > 1$ were first considered in \cite{KMS}, where some examples were constructed by computer (however, the related problem of $\lambda$-fold factorizations of groups has received some study; see \cite{YMSJ}).
Our current paper is the first to study \NF[(s,t)]{G}{\lambda} with $\lambda > 1$ in depth. One of the interesting aspects of our study is how different \NF[(s,t)]{G}{\lambda} with $\lambda > 1$ are, as compared to \NF[(s,t)]{G}{1}.

\subsection{Previous work and our contributions}

Before summarizing the contributions of our paper, we recall some of the many interesting published results concerning \NF[(s,t)]{G}{1}.
\begin{enumerate}
\item For any finite group $G$ and any $S \subseteq G$ with $|S| = s$, there is at most one set $T \subseteq G$ such that $(S,T)$ is an 
\NF[(s,t)]{G}{1} (where $st = |G|-1$). Moreover, there is an explicit formula to compute $T$ (if it exists) as a function of $S$  (see \cite{KMS}).
\item There has been considerable study of \NF[(s,t)]{G}{1}  in cyclic groups. Some relevant papers include
\cite{DB,CGHK,KPS,Nathanson}. 
\item There are also constructions for \NF[(s,t)]{G}{1} in dihedral groups; see \cite{BHS,CGHK,KPS,KPS2,Pech03,Pech}. 
\item There is no nontrivial \NF[(s,t)]{G}{1} in a noncyclic abelian group of order at most $200$; see \cite{CGHK,KMS,Pech}.
\item There are a couple of small sporadic examples of \NF[(s,t)]{G}{1} in nondihedral nonabelian groups; see \cite{KPS2,Pech}.
\item Every \NF[(s,t)]{G}{1} is equivalent to a symmetric \NF[(s,t)]{G}{1}  (see \cite{CGHK}; for the definitions of  symmetric near-factorizations and equivalence of near-factorizations, see Section \ref{symm.sec}).
\end{enumerate}

In this paper, we study \NF[(s,t)]{G}{\lambda} with $\lambda > 1$. The following are our main results.
 \begin{enumerate}
 \item For any finite group $G$ and any $S \subseteq G$ with $|S| = s$, there is at most one set $T \subseteq G$ such that $(S,T)$ is 
\NF[(s,t)]{G}{\lambda} (where $st = \lambda(|G|-1)$). Moreover, there is an explicit formula to compute $T$ (if it exists) as a function of $S$ (see Theorems \ref{T2.thm} and \ref{T3-l.thm}).
\item Some properties of symmetric \NF[(s,t)]{G}{\lambda} and equivalence of \NF[(s,t)]{G}{\lambda} are discussed in Section \ref{symm.sec}.
\item We prove some upper bounds on $\lambda$, assuming that an \NF[(s,t)]{G}{\lambda}  exists (see Section \ref{lambda.sec}).
\item We show that \NF[(s,t)]{G}{\lambda} with $s+t= |G|$ are equivalent to certain difference sets in $G$ (see Section \ref{DSandPDS.sec}). For most (but not all) parameter sets, the resulting near-factorizations are not equivalent to symmetric near-factorizations.
\item We show that \NF[(s,t)]{G}{\lambda} with $s+t+1= |G|$ are equivalent to certain partial difference sets in $G$  (see Section \ref{DSandPDS.sec}). The resulting near-factorizations are always symmetric.
\item We describe a product construction for certain near-factorizations in  Section \ref{prod.sec}.
\item We use the group ring to derive some restrictions on $s$ and $t$ for \NF[(s,t)]{G}{\lambda}, extending results from \cite{CGHK} (see Section \ref{subsetsize.sec}).
\item We use the computer to find some sporadic \NF[(s,t)]{G}{\lambda} in abelian groups that are not of the three types enumerated in items 3--5  above (see Section \ref{Comp.sec}). Examples of $\NF[(s,t)]{G}{\lambda}$ in abelian groups of small order are summarized in Tables~\ref{Mixed.tab} and~\ref{Pure.tab} (see Section \ref{Comp.sec}). 
\item We provide a list of instances of near-factorizations in abelian groups of small order, whose nonexistence only follows from exhaustive computer search (see Table~\ref{Computer.tab} in the Appendix). This table indicates that current theoretical nonexistence results for near-factorizations leave many parameter sets unresolved.
\end{enumerate}

\subsection{Some basic properties of near-factorizations}
\label{properties.sec}

We discuss a few basic properties of near-factorizations now. 
If $G$ is  any multiplicative group (not necessarily abelian), it is clear that 
$(ST)^{(-1)} = T^{(-1)}S^{(-1)}$ for any two subsets $S,T \subseteq G$.
Because $ST = \lambda (G - e)$ if and only if $(ST)^{(-1)} = \lambda (G -e)$, we see that
$ST = \lambda (G - e)$ if and only if $T^{(-1)}S^{(-1)} = \lambda (G - e)$. 
Hence, $(S,T)$ 
is an \NF[(s,t)]{G}{\lambda}
if and only if 
$(T^{(-1)},S^{(-1)})$ 
is a \NF[(t,s)]{G}{\lambda}.
Therefore we have the following result.

\begin{lemma} 
Suppose $G$ is any finite group. Then there is an \NF[(s,t)]{G}{\lambda}
if and only if there is a \NF[(t,s)]{G}{\lambda}. 
\end{lemma}

Suppose $|G| = n$ and the elements of $G$ are enumerated as
$g_1, \dots , g_n$. For convenience, we will assume that $g_1 = e$ (the identity).
For any subset $A \subseteq G$, define an $n \times n$ matrix $M(A)$ as follows
\[
(M(A))_{i,j} = \begin{cases}
1 & \text{if $g_i^{-1}g_j \in A$}\\
0 & \text{otherwise}.
\end{cases}
\]

Let $I$ denote the identity matrix and let $J$ be the all-ones matrix. The following important theorem holds for all finite groups (abelian and nonabelian).
This result was stated without proof in \cite{CGHK} in the case $\lambda = 1$.

\begin{theorem}
\label{equiv.thm}
Suppose $S,T \subseteq G$. Then $(S,T)$ 
is an \NF{G}{\lambda}
if and only if $M(S)M(T) = \lambda (J-I)$ \textup{(}this is an equation over the integers\textup{)}. 
\end{theorem}
\begin{proof}
Denote $M(S) = X = (x_{i,j})$ and $M(T) = Y = (y_{i,j})$. Let $XY = Z = (z_{i,j})$.
We have 
\begin{align*} z_{i,j} &= \sum_{k=1}^n x_{i,k} y_{k,j}\\
&= | \{ k : x_{i,k} = y_{k,j} = 1\} | \\
&= | \{ k : g_i^{-1}g_k \in S \text{ and } g_k^{-1}g_j \in T\} | \\
&= | \{ k : g_k \in g_i S \text{ and } g_k^{-1} \in Tg_j^{-1}\} | \\
&= | \{ k : g_k \in g_i S \text{ and } g_k \in g_j T^{(-1)}\} | \\
&= | g_i S \cap  g_j T^{(-1)} |.
\end{align*}

Suppose $h \in g_i S \cap  g_j T^{(-1)}$. Then
$h = g_i a = g_j b^{-1}$ for some $a \in S$, $b \in T$.
So $ab =  g_i^{-1} g_j$. If $i = j$, then $g_i^{-1} g_j = e$  and no such $x$ exists, because
$ab \neq e$ for any $a \in S$, $b \in T$. If $i \neq j$, then $g_i^{-1} g_j \neq e$ and there are exactly $\lambda$ pairs $(a,b)$ 
such that $ab = g_i^{-1} g_j$, $a \in S$ and $b \in T$.
It follows that
\[ z_{i,j} = \begin{cases}
0 & \text{if $i = j$}\\
\lambda & \text{if $i \neq j$.}
\end{cases}\]
This means that $XY = \lambda (J - I)$, as desired.
\end{proof}

The following two theorems are straightforward generalizations of results proven in \cite{KMS} in the case $\lambda = 1$.

\begin{theorem}
\label{T2.thm}
Let $\lambda$ be a positive integer and let $\lambda(J - I) = XY$ be a factorization into two $n \times n$ integral matrices $X$ and $Y$. Suppose also that $XJ = rJ$ for some positive integer $r$. Then
\[
 Y = \frac{\lambda}{r} J - \lambda X^{-1}.
\]
\end{theorem}

\begin{proof}
Because $\det(J-I) \neq 0$ and $\lambda(J-I) = XY$, it follows that $X$ and $Y$ are invertible.
Then we have the following:
\begin{align*}
Y &= X^{-1}\lambda(J-I)\\
&=  \lambda(X^{-1}J - X^{-1}) \\
&= \frac{\lambda}{r}X^{-1}(rJ)-\lambda X^{-1}\\
&=\frac{\lambda}{r}X^{-1}(XJ)-\lambda X^{-1} \\
&=\frac{\lambda}{r}J-\lambda X^{-1}.
\end{align*}
\end{proof}


Hence, we immediately obtain the following result as a corollary of Theorem \ref{equiv.thm}. Again, this result holds for all finite groups.

\begin{theorem}
\label{T3.thm}
Suppose $(S,T)$ and $(S,T')$ are both $\lambda$-fold near-factorizations of a finite group $G$. Then $T = T'$.
(Informally, if $S$ has a $\lambda$-mate $T$, then 
the $\lambda$-mate $T$ is unique.)
\end{theorem}

The following can be proven in a similar manner, by solving for $X$ as a function of $Y$:

\begin{theorem}
\label{T3a.thm}
Suppose $(S,T)$ and $(S',T)$ are both $\lambda$-fold near-factorizations of a finite group $G$. Then $S = S'$.
\end{theorem}

Theorem \ref{T3.thm} can be extended as follows.

\begin{theorem}
\label{T3-l.thm}
Suppose $(S,T)$ is a $\lambda$-fold near-factorization of a finite group $G$ and $(S,T')$ is a $\lambda'$-fold near-factorization of the same group $G$. Then $\lambda = \lambda'$ and $T = T'$.
\end{theorem}

\begin{proof}
Let $X = M(S)$, $Y = M(T)$ and $Y' = M(T')$. Then $X, Y$ and $Y'$ are 0-1 matrices such that $\lambda(J-I) = XY$, $\lambda' (J-I) = XY'$ and $XJ = sJ$. From Theorem \ref{T2.thm}, we have
\[ Y = \lambda \left( \frac{J}{s}- X^{-1}\right) \quad \text{and} \quad Y' = \lambda' \left( \frac{J}{s}- X^{-1}\right) .\]
Hence, 
\[ Y' = \frac{\lambda'}{\lambda} Y.\]
Because $Y$ and $Y'$ are 0-1 matrices, it must be the case that $\lambda = \lambda'$. Then $Y = Y'$ and hence $T = T'$.
\end{proof}

Theorem \ref{T3a.thm} can be extended in a similar fashion; the proof is omitted.

\begin{theorem}
\label{T3a-l.thm}
Suppose $(S,T)$ is a $\lambda$-fold near-factorization of a finite group $G$ and $(S',T)$ is a $\lambda'$-fold near-factorization of the same group $G$. Then $\lambda = \lambda'$ and $S = S'$.
\end{theorem}

\subsection{Symmetric near-factorizations}
\label{symm.sec}

A subset $A \subseteq G$ is \emph{symmetric} provided that  $g \in A$ if and only if $g^{-1} \in A$. In group ring notation, we have
$A = A^{(-1)}$. 
A near-factorization $(S,T)$ of a group $G$ is \emph{symmetric} if $S$ and $T$ are both symmetric.
The following lemma indicates that a near-factorization is symmetric if one of the two factors is symmetric.

\begin{lemma}
\label{SimpliesT.lem}
If $(S,T)$ is an \NF{G}{\lambda}, 
where one of the factors $S$ or $T$ is symmetric, then the other factor 
is also symmetric.
\end{lemma}

\begin{proof}
Suppose $ST = \lambda (G - e)$ and $S=S^{(-1)}$.
We compute $ T^{(-1)} ST$ in two ways. First,
\begin{align*}
T^{(-1)} ST &= T^{(-1)} S^{(-1)} T \\
&= (ST)^{(-1)} T \\
&=\lambda(G - e)^{(-1)} T \\
&=\lambda(G - e) T\\
&= \lambda t G - \lambda T.
\end{align*}
Second,
\begin{align*}
T^{(-1)} ST &= T^{(-1)} \lambda(G - e) \\
&= \lambda t G - \lambda T^{(-1)}.
\end{align*}
Equating the two expressions for $ T^{(-1)} ST$, we see that $T = T^{(-1)}$.

If $T=T^{(-1)}$, then we can prove $S=S^{(-1)}$ in a similar fashion.
\end{proof}


Now we introduce the notion of equivalence of subsets of a group $G$.
Let $(G,\cdot)$ be a multiplicative group and let 
$A, B \subseteq G$. 
We say that  $B$ is \emph{left equivalent} to $A$ if there exists $f \in \Aut(G)$ and $g\in G$ such that $B=gf(A)$.
Also, $B$ is \emph{right equivalent} to $A$ if there exists $f \in \Aut(G)$ and $g\in G$ such that $B=f(A)g$.

\begin{lemma}
$B$ is {left equivalent} to $A$ if and only if $A$ is {left equivalent} to $B$.
\end{lemma}

\begin{proof}
Suppose $B$ is {left equivalent} to $A$. Then there exists $f \in \Aut(G)$ and $g\in G$ such that $B=g f(A)$.
Hence, $g^{-1} B= f(A)$ and $f^{-1}(g^{-1} B) = A$.
Define $g^* = f^{-1}(g^{-1})$; then $A = g^* f^{-1}(B)$. Because $f^{-1} \in \Aut(G)$, we see that
$A$ is left equivalent to $B$. 
\end{proof}

\begin{lemma} Let $(G,\cdot)$ be a multiplicative group and let 
$A, B \subseteq G$. Then  $B$ is {left equivalent} to $A$ if and only if  $B$ is {right equivalent} to $A$.
\end{lemma}

\begin{proof}
Suppose $B$ is {left equivalent} to $A$. Then there exists $f \in \Aut(G)$ and $g\in G$ such that $B=g f(A)$.
Define $f^* \in \Aut (G)$ by $f^*(x) = g f(x) g^{-1}$ for all $x \in G$ (hence, $f(x) = g^{-1} f^*(x) g$). We have
\begin{align*}
B & =g f(A)\\
&= g (g^{-1} f^*(A) g)\\
& = f^*(A) g.
\end{align*}
Hence,  $B$ is {right equivalent} to $A$. The converse result is proven in a similar manner.
\end{proof}
Henceforth, we will just say that $A$ and $B$ are \emph{equivalent} if $B$ is left equivalent or right equivalent to $A$ or vice versa.

\medskip

Now we extend the equivalence notion from sets to near-factorizations.
Suppose $(S,T)$ and $(S',T')$ are both \NF[(s,t)]{G}{\lambda}.
We say that $(S,T)$ and $(S',T')$ are \emph{equivalent} near-factorizations if $S$ is equivalent to $S'$ and $T$ is equivalent to $T'$.

\begin{lemma}
\label{equivNF.lem}
Suppose  $(S,T)$ and $(S',T')$ are both \NF[(s,t)]{G}{\lambda}. Then  $(S,T)$ and $(S',T')$ are {equivalent} if and only if
there exists $f \in \Aut(G)$ and $g\in G$ such that $S' = f(S)g$ and $T' = g^{-1} f(T)$.
\end{lemma}

\begin{proof}
Suppose $f \in \Aut(G)$ and $g\in G$. We first show that $(f(S)g, g^{-1} f(T))$ is an \NF[(s,t)]{G}{\lambda}. We have 
\begin{align*}
(f(S)g)\,(g^{-1}f(T)) & =  f(S)f(T)\\
&=f(ST)\\
&=f(\lambda(G-e)) \\
&=\lambda(G-e).
\end{align*}
Now suppose that  $(S,T)$ and $(S',T')$ are equivalent near-factorizations. Because $S$ and $S'$ are equivalent, we must have $S' = f(S)g$ for some $f \in \Aut(G)$ and $g\in G$.
We have shown above that $S' (g^{-1}f(T)) = \lambda(G-e)$, so $g^{-1}f(T)$ is a mate to $S'$. By uniqueness of mates (Theorem \ref{T3.thm}), we must have $T' = g^{-1}f(T)$.
\end{proof}

In the case $\lambda = 1$, the following powerful result is known for abelian groups.

\begin{theorem}[{\cite[Proposition 2]{CGHK}}]
\label{symmetric.lem}
Suppose $G$ is an abelian group. 
If $(S,T)$ is an 
\NF[(s,t)]{G}{1}, 
then $(S,T)$ is equivalent to a symmetric \NF[(s,t)]{G}{1}.
\end{theorem}

Theorem  \ref{symmetric.lem} establishes that, up to equivalence, we only need to consider symmetric near-factorizations when $\lambda=1$ and the group $G$ is abelian. This simplifies  both theoretical analysis and computational searches. 
A natural question is whether Theorem \ref{symmetric.lem} can be extended to near-factorizations with $\lambda > 1$. Interestingly, the answer to this question is negative. In fact, we will show in Section \ref{DSandPDS.sec} that there are numerous examples of \NF[(s,t)]{G}{\lambda} with $\lambda > 1$ that are not equivalent to  symmetric near-factorizations. 

Actually, the stated equivalence is guaranteed in Theorem  \ref{symmetric.lem} even if we only consider equivalence via translations.  We define this concept now. Let $(G,\cdot)$ be a multiplicative group and let 
$A, B \subseteq G$. 
We say that  $B$ is \emph{left translate-equivalent} to $A$ if there exists $g\in G$ such that $B=gA$.
Also, $B$ is \emph{right translate-equivalent} to $A$ if there exists $g\in G$ such that $B=Ag$.
Because $B=gA$ if and only if $A = g^{-1}B$, it is immediate that
$B$ is {left translate-equivalent} to $A$ if and only if $A$ is {left translate-equivalent} to $B$.
Similarly, $B$ is {right translate-equivalent} to $A$ if and only if $A$ is {right translate-equivalent} to $B$.
Note that, in a nonabelian group, left equivalence does not necessarily imply right equivalence, and vice versa.




Suppose $(S,T)$ and $(S',T')$ are both \NF[(s,t)]{G}{\lambda}.
We say that the near-factorization $(S,T)$ is \emph{translate-equivalent} to the near-factorization $(S',T')$  if $S' = Sg$ and $T' = g^{-1} T$ for some $g \in G$.

\begin{theorem}
\label{transequiv.thm}
Suppose  $(S,T)$ is an \NF[(s,t)]{G}{\lambda}. Then $(S,T)$ is equivalent to a symmetric \NF[(s,t)]{G}{\lambda} if and only if 
$(S,T)$ is translate-equivalent to a symmetric \NF[(s,t)]{G}{\lambda}.
\end{theorem}

\begin{proof}
Suppose $(S,T)$ is equivalent to a  symmetric \NF[(s,t)]{G}{\lambda}, say $(S',T')$.
Because $(S,T)$ is equivalent to $(S',T')$, it follows  from Lemma \ref{equivNF.lem}
that there exists $f \in \Aut(G)$ and $g\in G$ such that $S' = f(S)g$ and $T' = g^{-1} f(T)$.
Define $S'' = f^{-1}(S')$  and $T'' = f^{-1}(T')$. Clearly $S''$ and $T''$ are both symmetric.
We also have \[S'' = f^{-1}(f(S)g) = S g^*,\] where $g^* = f^{-1}(g)$. Similarly, \[T'' = f^{-1}(g^{-1}f(T)) = (g^*)^{-1} T.\]


Further, it is easy to see that
$(S'', T'')$ is an \NF[(s,t)]{G}{\lambda}, because
\begin{align*} S''T'' &= f^{-1}(S')f^{-1}(T') \\&= f^{-1}(S'T') \\&= f^{-1}(\lambda(G - e)) \\&= \lambda(G - e).\end{align*}
Hence, $(S,T)$ is translate-equivalent to $(S'',T'')$, which is a symmetric \NF[(s,t)]{G}{\lambda}.

The converse assertion is obviously true.
\end{proof}

\subsection{Upper bounds on $\lambda$}
\label{lambda.sec}

We will prove some upper bounds on $\lambda$ in this section. First, we prove some preliminary lemmas.

\begin{lemma}
\label{Sinverse.lem}
Suppose $G$ is a group of order $n$ and $(S,T)$ is an 
\NF[(s,t)]{G}{\lambda}.  Then $S^{(-1)} \subseteq G \setminus T$.
\end{lemma}
\begin{proof}
If $x \in S$, then 
$x^{-1} \not\in T$. Hence, $S \cap T^{(-1)} = \emptyset$.
Equivalently, $S^{(-1)} \subseteq G \setminus T$.
\end{proof}

\begin{theorem}
\label{sum.thm}
Suppose $G$ is a group of order $n$. 
If there is an 
\NF[(s,t)]{G}{\lambda}, then $s + t \leq n$ and 
 \begin{equation}
\label{st.eq} \lambda \leq \frac{st}{s+t-1}.
\end{equation}
\end{theorem}

\begin{proof}
Suppose $(S,T)$ is an 
\NF[(s,t)]{G}{\lambda}. We have
$ |S^{(-1)}| = |S| = s$.
Then, becuase $S^{(-1)} \subseteq G \setminus T$, we have $s \leq n - t$ and hence $s + t \leq n$.
To prove that (\ref{st.eq}) holds, we observe that
\begin{align*}
s + t \leq n &\Leftrightarrow s+t-1 \leq n-1\\
&\Leftrightarrow \lambda(s+t-1) \leq \lambda(n-1)\\
&\Leftrightarrow \lambda(s+t-1) \leq st\\
&\Leftrightarrow \lambda \leq \frac{st}{s+t-1},
\end{align*}
so (\ref{st.eq}) holds.
\end{proof}

\begin{corollary}
\label{upperbound.cor}
 If $\min \{s,t\} > 1$ and there exists an \NF[(s,t)]{G}{\lambda}, then $\lambda \leq \min \{s-1,t-1\}$.
\end{corollary}

\begin{proof}
We have  $\lambda \leq {st}/(s+t-1)$ from (\ref{st.eq}). Because $t > 1$, we have $s < s+t-1$ and hence
\[ \lambda \leq \frac{st}{s+t-1} < t.\] 
Therefore, $\lambda \leq t-1$. We also have $s > 1$, which implies $\lambda \leq s-1$ by a similar argument.
\end{proof}

\begin{theorem}
\label{upperbound.lem}
Suppose $G$ is a group of order $n$. 
If there is an 
\NF[(s,t)]{G}{\lambda}, 
then 
\[
\lambda \leq  
\left\lfloor \frac{s+t+1}{4} \right\rfloor. 
\]
\end{theorem}

\begin{proof} 
Denote $r = s+t$. 
First, suppose $r$ is even. From (\ref{st.eq}), we have
\begin{align*}
\lambda & \leq \frac{(\frac{r}{2})^2}{r-1}\\
&= \frac{r^2}{4(r-1)}\\
&= \frac{r}{4} + \frac{1}{4} + \frac{1}{4(r-1)}.
\end{align*}
We have $st \geq 2$, so $r > 2$. Then we have 
 \[ \frac{1}{4} + \frac{1}{4(r-1)} < \frac{1}{4} + \frac{1}{4} = \frac{1}{2},\]
so \[\lambda <  \frac{r}{4} + \frac{1}{2}.\]
But $r$ is an even integer, so 
\[\lambda \leq  \left\lfloor \frac{r}{4}\right\rfloor = \left\lfloor \frac{r+1}{4}\right\rfloor.\]

Now we suppose $r$ is odd. From (\ref{st.eq}), we have 
\begin{align*}
\lambda  &\leq \frac{\left( \frac{r-1}{2}\right)\left( \frac{r+1}{2}\right)}{r-1}\\
&= \frac{r^2-1}{4(r-1)}\\
&= \frac{r+1}{4}.
\end{align*}
Because $\lambda$ is an integer, we have 
\[ \lambda \leq \left\lfloor \frac{r+1}{4}\right\rfloor.\]
\end{proof}

We will prove later that the bound from Theorem \ref{upperbound.lem} is sometimes tight (see Corollary \ref{cor2.21}).

\medskip
Because $s + t \leq n$ (Theorem \ref{sum.thm}), we immediately obtain the following corollary that upperbounds $\lambda$ in terms of $n$.

\begin{corollary}
\label{nbound.cor}
Suppose $G$ is a group of order $n$. 
If there is an 
\NF[(s,t)]{G}{\lambda}, 
then \[
\lambda \leq  
\left\lfloor \frac{n+1}{4} \right\rfloor. \]
\end{corollary}

\medskip

%
%
%
Here are two further simple results that also follow from Theorem \ref{sum.thm}.

\begin{lemma}
\label{11.lem}
Suppose $G$ is a group of order $n$. Then there exists an \NF[(s,n-1)]{G}{\lambda} if and only if $\lambda = s = 1$.
\end{lemma}

\begin{proof}
Suppose $(S,T)$ is an \NF[(s,n-1)]{G}{\lambda}.
From Theorem \ref{sum.thm}, we have $s = 1$. We also have $\lambda (n-1) = 1 \times (n-1)$ from (\ref{prod.eq}), so $\lambda = 1$.
Finally, we observe that a \NF[(1,n-1)]{G}{1} always exists by taking $S = e$ and $T = G - e$.   
\end{proof}

\begin{lemma}
\label{n-1prime.lem}
Suppose $G$ is a group of order $n$, where $n-1$ is prime. Then there exists an \NF[(s,t)]{G}{\lambda} if and only if $\lambda = 1$ and
$\{s,t\} = \{1,n-1\}$.
\end{lemma}

\begin{proof}
We have $st = \lambda (n-1)$ from (\ref{prod.eq}). Because $n-1$ is prime, one of $s$ or $t$ is divisible by $n-1$. Without loss of generality, suppose $s$ is divisible by $n-1$. Then $s \geq n-1$. Because $s + t \leq n$ (Theorem \ref{sum.thm}), we must have $s = n-1$ and $t =1$, so $\lambda = 1$ from (\ref{prod.eq}).
\end{proof}

Our next theorem sometimes improves the bound on $\lambda$ as a function of $n$ given in Corollary \ref{nbound.cor}, depending on how $n$ can be factored.

\begin{theorem}
\label{n1n2.thm}
Suppose $G$ is a group of order $n$, where $n-1 = n_1n_2$ for positive integers $n_1$ and $n_2$.
Suppose there exists an \NF[(s,t)]{G}{\lambda} with $\min \{ s,t\} > 1$, $s = \lambda_1 n_1$, $t = \lambda_2 n_2$ and
$\lambda = \lambda_1\lambda_2$. 
Then
\begin{equation}
\label{nbound.eq}
\lambda \leq 
\begin{cases}
\frac{n+1}{4} & \text{if $n_1$ and $n_2$ are both even}\\
\frac{(n_1+1)(n - n_2 + 1)}{4n_1} & \text{if $n_1$ is odd and  $n_2$ is even}\\
\frac{(n_2+1)(n - n_1 + 1)}{4n_2} & \text{if $n_1$ is even and  $n_2$ is odd}\\
\min \left\{ \frac{(n_1+1)(n - n_2 + 1)}{4n_1}, \frac{(n_2+1)(n - n_1 + 1)}{4n_2} \right\} & \text{if $n_1$ and $n_2$ are both odd.}
\end{cases}
\end{equation}
\end{theorem}

\begin{proof}


We  have
\[ s+t = \lambda_1n_1 + \lambda_2n_2 \leq n,\]
so \[ \lambda_1 \leq \frac{n - \lambda_2n_2}{n_1}.
\]
Therefore,
\[ \lambda = \lambda_1 \lambda_2 \leq  \lambda_2 \left( \frac{n - \lambda_2n_2}{n_1}\right).\]

Define $f(x) = x (n - x n_2)/n_1$. We have $f'(x) = (n - 2x n_2)/n_1$, so $f'(x) = 0$ when 
$x = n/(2n_2)$. Note that
\[ \frac{n}{2n_2} = \frac{n_1n_2+1}{2n_2} = \frac{n_1}{2} + \frac{1}{2n_2}.\]
Because $x$ must be an integer, we see that $f(x)$ is maximized when $x$ is the nearest integer to $\frac{n_1}{2} + \frac{1}{2n_2}$; denote this integer by $x_0$.

If $n_1$ is even, then $x_0 = n_1 / 2$ and
\[f(x_0) = \left(\frac{n_1}{2}\right) \frac{\left(n - \left(\frac{n_1}{2}\right) n_2\right)}{n_1} = \frac{n+1}{4}.\]

If $n_1$ is odd, then $x_0 = (n_1+1) / 2$ and
\[f(x_0) = \left(\frac{n_1+1}{2}\right) \frac{\left(n - \left(\frac{n_1+1}{2}\right) n_2\right)}{n_1} = 
\frac{(n_1+1)(n - n_2+1)}{4n_1}.\]
Hence,
\begin{equation}
\label{n_1bound.eq}
\lambda \leq 
\begin{cases}
\frac{n+1}{4} & \text{if $n_1$ is even}\\
\frac{(n_1+1)(n - n_2 + 1)}{4n_1} & \text{if $n_1$ is odd.}
\end{cases}
\end{equation}
Interchanging the roles of $n_1$ and $n_2$, we have
\begin{equation}
\label{n_2bound.eq}
\lambda \leq 
\begin{cases}
\frac{n+1}{4} & \text{if $n_2$ is even}\\
\frac{(n_2+1)(n - n_1 + 1)}{4n_2} & \text{if $n_2$ is odd.}
\end{cases}
\end{equation}
Combining (\ref{n_1bound.eq}) and (\ref{n_2bound.eq}), we obtain (\ref{nbound.eq}).
\end{proof}

As an example to illustrate the application of  Theorem \ref{n1n2.thm}, we show that the general upper bound on $\lambda$ from Corollary \ref{nbound.cor} (which is roughly $n/4$) can be improved to roughly $2n/9$ when $n \equiv 4 \bmod 6$.

\begin{corollary}
\label{3p.lem}
Suppose $G$ is a group of order $n$, where $n \equiv 4 \bmod 6$. Then there exists an \NF[(s,t)]{G}{\lambda} only if $\lambda \leq (2n+4)/9$.
\end{corollary}

\begin{proof}Apply Theorem \ref{n1n2.thm} with $n_1 = 3$ and  $n_2 = (n-1)/3$.
\end{proof}

\section{Constructions of near-factorizations}

In this section, we 
discuss how certain difference sets and partial difference sets give rise to near-factorizations and we describe a ``product''construction for NFs.

\subsection{Constructions from difference sets and partial difference sets}
\label{DSandPDS.sec}

We now describe a construction for near-factorizations from certain difference sets. 
First, we define a \emph{$(v,k,\lambda)$-difference set} (or for short, \emph{$(v,k,\lambda)$-DS}) in a group $G$ of order $v$ to be a subset $D \subseteq G$ such that $|D| = k$ and the following group ring equation is satisfied:
\begin{equation}
\label{DS.eq}
 D D^{(-1)} = k \cdot e + \lambda (G - e),
\end{equation}
where $e$ is the identity element in $G$.
Note that a necessary condition for existence of a $(v,k,\lambda)$-DS is that $\lambda(v-1) = k(k-1)$. 

In any group $G$ of order $v$, there always exists a trivial $(v,1,0)$-DS in $G$ (namely, $\{e\}$) and a trivial $(v,v-1,v-2)$-DS  in $G$ (namely, $G \setminus \{e\}$). A $(v,k,\lambda)$-DS in $G$ is \emph{nontrivial} if with $1 < k < v-1$. We note that if $D$ is a $(v,k,\lambda)$-DS in $G$, then so is $D^{(-1)}=\{ d^{-1} : d \in D \}$. Moreover, $G \setminus D$ is called the \emph{complement} of $D$ in $G$; it is easy to see that $G \setminus D$ is a $(v,v-k,v-2k+\lambda)$-DS in $G$. For a comprehensive treatment of difference sets, please refer to \cite[Chapter VI]{BJL} and \cite{JPS}.

The term ``difference set'' implicitly assumes that the group $G$ is written additively. In this situation, if we compute all the differences of distinct elements in $D$, we get every nonzero group element exactly $\lambda$ times. 
Nevertheless, we will use multiplicative notation in most of our theorems. 

\begin{theorem}
\label{differencesets.thm}
Suppose there is a $(v,k,\lambda)$-DS $D$ in a group $G$ of order $v$.
Then $(D, G - D^{(-1)})$ is a  $\NF[(k,v-k)]{G}{k-\lambda}$.
\end{theorem}
\begin{proof}
Let $D$ be a $(v,k,\lambda)$-DS in the group $G$. We perform computations in the group ring
$\zed[G]$. Define group ring elements $S = D$ and $T = G - D^{(-1)}$.  
We will compute the product $ST$. First, it is clear that 
\[ SG = kG.\]
Next, 
\[ S D^{(-1)} = D  D^{(-1)} = k \cdot e  + \lambda (G - e).\]
Hence,
\begin{align*}
 S T &= S (G - D^{(-1)})\\
 &= kG - (k \cdot e  + \lambda (G - e))\\ 
&= (k - \lambda )(G - e) .
\end{align*}
Therefore, $(S,T)$ is the desired near-factorization.
\end{proof}

\begin{example}
{\rm $D = \{1,3,4,5,9\}$ is an $(11,5,2)$-DS in $(\zed_{11},+)$. We are working in an additive group, so
\[D^{(-1)} = \{-1,-3,-4,-5,-9\} = \{2,6,7,8,10\}\] and  
\[ T = \zed_{11} \setminus D^{(-1)} = \{0,1,3,4,5,9\}.\] From Theorem \ref{differencesets.thm}, we see that
$D$ and $T$ yield a \NF[(5,6)]{\zed_{11}}{3}.}$\hfill\blacksquare$ 
\end{example}

We can also prove a converse result to Theorem \ref{differencesets.thm}.

\begin{theorem}
\label{DSconverse.thm}
Suppose $(S,T)$ is an $\NF[(s,t)]{G}{\lambda}$, where $|G| = s + t$.
Then $S$ is an $(s+t,s,s-\lambda)$-DS in $G$ and 
$T = G - S^{(-1)}$  is an $(s+t,t,t -\lambda)$-DS in $G$.
\end{theorem}

\begin{proof}
We again work in the group ring $\zed[G]$. 
We have $S^{(-1)} = G - T$, so
\begin{align*} S S^{(-1)} &= S(G - T)\\ 
&= sG - ST 
\\&= sG - \lambda (G - e) 
\\&= s(G - e) + s\cdot e - \lambda (G - e) \\&= (s-\lambda)(G-e) + s \cdot e.\end{align*}
Hence, $S$ is an  $(s+t,s,s-\lambda)$-DS in $G$.

Therefore, $S^{(-1)}$ is an $(s+t,s,s-\lambda)$-DS in $G$. As $T = G \setminus S^{(-1)}$ is the complement of $S^{(-1)}$, 
$T$ is an $(s+t,t,t-\lambda)$-DS in $G$.
\end{proof}

Combining Theorems~\ref{differencesets.thm} and~\ref{DSconverse.thm}, we have the following characterization of $\NF[(s,t)]{G}{\lambda}$ with $|G|=s+t$.
\begin{corollary}
\label{DScharacterization.cor}
Let $G$ be a finite group. Then $(S,T)$ is an $\NF[(s,t)]{G}{\lambda}$ with $|G|=s+t$ if and only if $S$ is an $(s+t,s,s-\lambda)$-DS in $G$. 
\end{corollary}

A difference set $D$ in group $G$ is \emph{reversible} if $D = D^{(-1)}$. 
Applying Theorems~\ref{differencesets.thm} and~\ref{DSconverse.thm}, we have the following characterization of symmetric $\NF[(s,t)]{G}{\lambda}$ with $|G|=s+t$.

\begin{corollary}
\label{revDScharacterization.cor}
Let $G$ be a finite group.  Then $(S,T)$ is a symmetric $\NF[(s,t)]{G}{\lambda}$ with $|G|=s+t$ if and only if $S$ is a reversible $(s+t,s,s-\lambda)$-DS in $G$. 
\end{corollary}

A difference set in an abelian group is called an \emph{abelian difference set}. Reversible abelian difference sets have been extensively studied (see \cite[Chapter VI, Section 14]{BJL} and \cite[Section 12]{Ma}). 
Indeed, McFarland has conjectured that a reversible abelian difference set has parameters
$(4u^2,2u^2 \pm u, u^2 \pm u)$ for some integer $u \geq 1$; or parameters $(4000, 775, 150)$ or
$(4000,3225, 2600)$ (see \cite[Chapter VI, Conjecture 14.38]{BJL}).
The difference sets with parameters $(4u^2,2u^2 \pm u, u^2 \pm u)$ are known as \emph{Menon difference sets} or, alternatively,
\emph{Hadamard difference sets}. Below, we summarize the existence status of reversible abelian difference sets.

\begin{theorem} 
\label{revDS.prop}
Up to complementation, the following are all known reversible abelian difference sets. 
\begin{enumerate}
\item {\rm (\cite[Chapter VI, Theorem 14.46]{BJL} and \cite[Example 12.8]{Ma})} Let $G$ be a group of the form 
\[G=(\zed_2)^{2a} \times (\zed_4)^b \times (\zed_{2^{r_1}})^{2c_1} \times \cdots \times (\zed_{2^{r_s}})^{2c_s} \times (\zed_3)^{2d} \times (\zed_{p_1})^{4e_1} \times \cdots \times (\zed_{p_t})^{4e_t},\] where $r_j \ge 1$ and the $p_j$'s are not necessarily distinct odd primes and $a, b, c_1, \ldots, c_s, d, e_1,\ldots,e_t$ are non-negative integers satisfying
\begin{enumerate}
\item  $a+b+c_1+\cdots+c_s>0$.
\item  $a>0$ if $d+e_1+\cdots+e_t>0$. 
\end{enumerate}
Then $G$ contains a reversible $(4u^2,2u^2-u,u^2-u)$ Hadamard difference sets with $|G|=4u^2$. 
\item  {\rm (\cite[Chapter VI, Examples 14.37(a)]{BJL} and \cite[Example 12.4(4)]{Ma})} A reversible $(4000, 775, 150)$-DS in $(\zed_2)^5 \times (\zed_5)^3$.
\end{enumerate}
\end{theorem}

Infinite families of symmetric $\NF[(s,t)]{G}{\lambda}$ with $|G|=s+t$ follow from Corollary~\ref{revDScharacterization.cor} and Theorem~\ref{revDS.prop}.  

\begin{corollary}
Symmetric (s,t)-near-factorizations of group $G$ with $|G|=s+t$ exist in the following cases.
\begin{enumerate}
\item Symmetric $\NF[(2u^2-u,2u^2+u)]{G}{u^2}$, where $|G|=4u^2$ and $G$ satisfies the condition in Theorem~\ref{revDS.prop}, part 1.
\item Symmetric $\NF[(775,3225)]{(\zed_2)^5 \times (\zed_5)^3}{625}$.
\end{enumerate}
\end{corollary}

\begin{example}
{\rm Note that $D = \{(0,1), (1,0), (1,1), (0,3), (3,0), (3,3)\}$ is a reversible $(16,6,2)$ Hadamard DS in $(\zed_4)^2$.
Hence, Corollary \ref{revDScharacterization.cor} yields a symmetric $\NF[(6,10)]{(\zed_4)^2}{4}$.
}$\hfill\blacksquare$
\end{example}

On the other hand, Corollary \ref{revDScharacterization.cor} indicates a possible method to construct an $\NF[(s,t)]{G}{\lambda}$ with $|G|=s+t$ that is not equivalent to any symmetric near-factorization.

\begin{theorem}
\label{nonsymNF.cor}
Let $G$ be a finite group. Let $S$ be an $(s+t,s,s-\lambda)$-DS in $G$. Suppose that $G$ does not admit a reversible $(s+t,s,s-\lambda)$-DS. Then $(S,G-S^{(-1)})$ is an $\NF[(s,t)]{G}{\lambda}$ that is not equivalent to any symmetric near-factorization. 
\end{theorem}
\begin{proof}
Suppose that $(S,G-S^{(-1)})$ is equivalent to a symmetric $\NF[(s,t)]{G}{\lambda}$. From  Theorem \ref{transequiv.thm}, $(S,G-S^{(-1)})$ is translate-equivalent to a symmetric $\NF[(s,t)]{G}{\lambda}$, say $(S',T')$. Because $S$ is an $(s+t,s,s-\lambda)$-DS in $G$ and $S' = Sg$ for some $g \in G$, it follows that $S'$ is also an
$(s+t,s,s-\lambda)$-DS in $G$. However, from Corollary \ref{revDScharacterization.cor}, $(S',T')$ is symmetric, so $S'$ must be reversible. This contradicts the assumption that a reversible $(s+t,s,s-\lambda)$-DS in $G$ does not exist.
\end{proof}

\begin{remark}
{\rm \mbox{\quad}\vspace{-.2in}\\
\begin{enumerate}
\item There are  triples $(G,s,t)$ such that every $(s,t)$-near-factorization in $G$ is necessarily symmetric. By Theorem~\ref{revDS.prop}(1), for $a \ge 1$, each $(2^{2a+2},2^{2a+1}-2^a,2^{2a}-2^a)$ Hadamard DS in $(\zed_{2})^{2a+2}$ is reversible. By Corollary~\ref{revDScharacterization.cor}, each $\NF[(2^{2a+1}-2^a,2^{2a+1}+2^a)]{(\zed_{2})^{2a+2}}{2^{2a}}$ is symmetric.
\item There are  triples $(G,s,t)$ such that no $(s,t)$-near-factorization in $G$ is symmetric.
By \cite[Chapter VI, Theorem 14.39.g]{BJL}, a reversible $(v,k,\lambda)$ abelian DS must have parameters 
\begin{equation}
\label{revDS.eq}
\begin{split}
     v&= m(m+a+1)(m+a-1)/a, \\
     k&= m(m+a), \\
     \lambda&= ma,
\end{split}
\end{equation}
where $a$ and $m$ are positive integers having opposite parity, and $a$ divides $m^2-1$. 
Consequently, in view of Theorem \ref{nonsymNF.cor}, each $(v,k,\lambda)$-DS in an abelian group $G$ whose parameters do not satisfy the above conditions necessarily leads to a $\NF[(k,v-k)]{G}{k-\lambda}$ that is not symmetric. Indeed, by Corollary~\ref{revDScharacterization.cor}, if a group $G$ admits $(v,k,\lambda)$-DS but no reversible $(v,k,\lambda)$-DS, then every $\NF[(k,v-k)]{G}{k-\lambda}$ is not symmetric. Specifically, we have the following.
\begin{enumerate}
\item From \cite[Chapter VI, Theorem 14.39.a]{BJL}, each abelian group $G$ containing a $(v,k,\lambda)$-DS with $v$ being odd does not admit a reversible $(v,k,\lambda)$-DS. Therefore, $G$ admits a $\NF[(k,v-k)]{G}{k-\lambda}$ but no symmetric $(k,v-k)$-near-factorizations.
\item Let $G$ be an abelian $2$-group with order $2^{2a+2}$ and exponent $2^{a+2}$ for $a \ge 1$. Then by \cite[Theorem 1.2, Theorem 3.5]{DJ}, $G$ contains a $(2^{2a+2},2^{2a+1}-2^a,2^{2a}-2^a)$ Hadamard difference set but no reversible Hadamard difference set with such parameters. Therefore, $G$ admits a $\NF[(2^{2a+1}-2^a,2^{2a+1}+2^a)]{G}{2^{2a}}$ but no symmetric $(2^{2a+1}-2^a,2^{2a+1}+2^a)$-near-factorization.
\end{enumerate}
\item There are  triples $(G,s,t)$ such that $G$ admits both a symmetric $(s,t)$-near-factorization and an $(s,t)$-near-factorization that is not equivalent to any symmetric near-factorization. Specifically, we have the following.
\begin{enumerate}
\item For $(G,s,t)=(\zed_4 \times (\zed_2)^2,6,10)$: 

$D_1=\{ (0,0,0), (0,0,1), (0,1,0), (1,0,0), (2,1,1), (3,0,0) \}$ is a reversible $(16,6,2)$ Hadamard difference set in $G$. By Corollary \ref{revDScharacterization.cor}, $(D_1,G-D_1^{(-1)})$ is a symmetric $\NF[(6,10)]{G}{4}$.

$D_2=\{ (0,0,0), (0,0,1), (0,1,0), (1,0,0), (2,0,0), (3,1,1) \}$ is a $(16,6,2)$ Hadamard difference set in $G$. It can be verified that $D_2$ is not translate-equivalent to a symmetric $(16,6,2)$-DS in $G$. Hence, 
from Theorem \ref{transequiv.thm}  and  Corollary \ref{revDScharacterization.cor}, 
$(D_2,G-D_2^{(-1)})$ is a $\NF[(6,10)]{G}{4}$ that is not equivalent to any symmetric $(6,10)$-near-factorization.

\item For $(G,s,t)=((\zed_4)^2,6,10)$: 

$D_1=\{ (1,0), (1,1), (2,0), (2,2), (3,0), (3,3) \}$ is a reversible $(16,6,2)$ Hadamard difference set in $G$. By Corollary \ref{revDScharacterization.cor}, $(D_1,G-D_1^{(-1)})$ is a symmetric $\NF[(6,10)]{G}{4}$.

$D_2=\{ (0,0), (0,1), (0,3), (1,0), (2,0), (3,2) \}$ is a $(16,6,2)$ Hadamard difference set in $G$. It can be verified that $D_2$ is not translate-equivalent to a symmetric $(16,6,2)$-DS in $G$. Hence, 
from Theorem \ref{transequiv.thm}  and  Corollary \ref{revDScharacterization.cor}, 
$(D_2,G-D_2^{(-1)})$ is a $\NF[(6,10)]{G}{4}$ not equivalent to any symmetric $(6,10)$-near-factorization.
\end{enumerate}
We note that the above difference sets in $\zed_4 \times (\zed_2)^2$ and $(\zed_4)^2$ are derived from the \emph{La Jolla Difference Sets Repository} maintained by Gordon \cite{Gordon}. $\hfill\blacksquare$
\end{enumerate}
}
\end{remark}

Corollaries \ref{DScharacterization.cor} and \ref{revDScharacterization.cor} provide nice characterizations of $\NF[(s,t)]{G}{\lambda}$ and symmetric $\NF[(s,t)]{G}{\lambda}$ with $|G|=s+t$ in terms of difference sets and reversible difference sets. We now examine $\NF[(s,t)]{G}{\lambda}$ with $|G|=s+t+1$. These will turn out to be equivalent to certain partial difference sets.

A \emph{$(v, k, \lambda, \mu)$-partial difference set} (or for short, \emph{$(v, k, \lambda, \mu)$-PDS}) in a group $G$ of order $v$ is  a subset $D \subseteq G \setminus \{e\}$ such that $|D| = k$ and the following group ring equation is satisfied:
\[
 D D^{(-1)} = (k - \mu) \cdot e + (\lambda - \mu) D + \mu G,
\]
where $e$ is the identity element in $G$.
That is, when we compute the ``differences'' of elements in $D$, there are exactly $\lambda$ occurrences of each  element of $D$, and exactly $\mu$ occurrences of each nonidentity element of $G \setminus D$.
Note that a PDS with $\lambda = \mu$ is a difference set. We will be considering PDS with $\lambda \neq \mu$ in the rest of this section.  

Suppose $D$  is a $(v, k, \lambda, \mu)$-PDS in $G$. Then $G \setminus (D \cup \{e\})$ is called the \emph{complement} of $D$ in $G$. It is straightforward to see that $G \setminus (D \cup \{e\})$  is a $(v,v-k-1,v-2-2k+\mu,v-2k+\lambda)$-PDS in $G$. 

\begin{example}
\label{13.ex}
{\rm The set $\{1,3,4,9,10,12\}$ is a  $(13, 6, 2, 3)$-PDS in $(\zed_{13},+)$.}$\hfill\blacksquare$
\end{example}

\begin{theorem}
\label{PDStoNF.thm}
Suppose $D$ is a $(s+t+1, s, s-\lambda-1, s-\lambda)$-PDS in a group $G$. 
Then $(S,T)$ is an \NF[(s,t)]{G}{\lambda}, where $S = D$ and $T = G - D^{(-1)} - e$.
\end{theorem}

\begin{proof}
Let $S$ and $T$ be defined as in the statement of the theorem. We have
\begin{align*}
 ST &= DG - DD^{(-1)} - D \\
 &= sG - ((s - (s - \lambda)) \cdot e + (-1) D + (s - \lambda) G) - D\\
&= sG - (\lambda \cdot e - D + ( s-\lambda) G) - D\\
&= \lambda ( G - e).
\end{align*}
Therefore $(S,T)$ is an \NF[(s,t)]{G}{\lambda}.
\end{proof}

\begin{example}
{\rm From the $(13, 6, 2, 3)$-PDS given in Example \ref{13.ex}, we obtain a \NF[(6,6)]{\zed_{13}}{3}.
The near-factorization has $S = \{1,3,4,9,10,12\}$ and $T = \{ 2,5,6,7,8,11\}$.}$\hfill\blacksquare$
\end{example}

Now we prove a converse result.

\begin{theorem}
\label{NFtoPDS.thm}
Suppose $(S,T)$ is an \NF[(s,t)]{G}{\lambda}, where $|G| = s+t+1$. Then $S$ is 
translate-equivalent to an $(s+t+1, s, s-\lambda-1, s-\lambda)$-PDS in $G$ and $T$ is  translate-equivalent to
an $(s+t+1, t, t-\lambda-1, t-\lambda)$-PDS.
\end{theorem}

\begin{proof}
From Lemma \ref{Sinverse.lem}, we have $S^{(-1)} \subseteq G \setminus T$. Equivalently, 
$T \subseteq G \setminus S^{(-1)}$. We have $|G| = s+t+1$, $|S^{(-1)}| = |S| = s$ and $|T| = t$, so it follows that
$G \setminus (S^{(-1)} \cup T) = \{g\}$ for some $g \in G$. 

If $g \neq e$, replace $S$ by $S' = Sg$ and replace $T$ by $T' = g^{-1}T$. Then 
$S'T' = G - e$. We have $(S')^{(-1)} = (Sg)^{(-1)} = g^{-1}(S)^{(-1)}$ and 
\begin{align*}
G \setminus ((S')^{(-1)} \cup T') &= G \setminus (g^{-1}(S)^{(-1)} \cup g^{-1}T) \\
&= g^{-1}G \setminus (g^{-1}(S)^{(-1)} \cup g^{-1}T) \\
&= g^{-1}(G \setminus ((S)^{(-1)} \cup T) \\
&= \{g^{-1}g\}\\
&= \{e\}.
\end{align*}

Therefore, from now on, we assume that  $G \setminus (S^{(-1)} \cup T) = \{e\}$.
Because we also have $T \subseteq G \setminus S^{(-1)}$, it follows that $S^{(-1)}=G-T-e$.
We compute 
\begin{align*}
SS^{(-1)} & = S (G - T - e)\\
&= SG - ST - Se\\
&= sG - \lambda(G - e) - S\\
&= \lambda \cdot e + (s - \lambda)G - S\\
&= s \cdot e + (s-\lambda-1)S + (s-\lambda)(G-S-e).
\end{align*}
Therefore, $S$ is an $(s+t+1, s, s-\lambda-1, s-\lambda)$-PDS in $G$.
Similarly, $T$ is an $(s+t+1, t, t-\lambda-1, t-\lambda)$-PDS in $G$.
\end{proof}

If $D$ is a $(v, k, \lambda, \mu)$-PDS where $\lambda \neq \mu$, then $D = D^{(-1)}$ (see, for example,
\cite[Proposition 1.2]{Ma}). Combining this fact with Lemma \ref{SimpliesT.lem} and Theorems~\ref{PDStoNF.thm}, ~\ref{NFtoPDS.thm}, we have the following characterization of $\NF[(s,t)]{G}{\lambda}$ with $|G|=s+t+1$.
\begin{corollary}
\label{PDScharacterization.cor}
Let $G$ be a finite group. Then $(S,T)$ is an $\NF[(s,t)]{G}{\lambda}$ with $|G|=s+t+1$ if and only if $S$ is  translate-equivalent to an $(s+t+1, s, s-\lambda-1, s-\lambda)$-PDS in $G$. Moreover, $(S,T)$ is necessarily  translate-equivalent to a symmetric near-factorization.  
\end{corollary}

In Corollary \ref{PDScharacterization.cor}, the two factors $S$ and $T$ are both partial difference sets with $\lambda-\mu=-1$. Note that a partial difference set in an abelian group is called an \emph{abelian partial difference set}. Abelian partial difference sets with $\lambda-\mu=-1$ have been extensively studied. Their parameters have been classified in \cite[Theorem 1.1]{AJMP} and \cite[Theorem 1.2]{Wang} and all known constructions are summarized below.

\begin{theorem}
\label{PDScharacterization.prop}
Let $D$ be an abelian $(v,k,\lambda,\mu)$-PDS with $\lambda-\mu=-1$ in a group $G$. Then up to complementation, one of the following holds.
\begin{enumerate}
\item $D$ is a Paley type PDS with parameter $(v,k,\lambda,\mu)=(v, (v-1)/2, (v- 5)/4, (v- 1)/4)$ for some $v \equiv 1 \bmod 4$, where one of the following conditions is satisfied:
\begin{enumerate}
\item $v$ equals a prime power $p^s$ with $p^s \equiv 1 \bmod 4$. Such partial difference sets exist in $(\zed_p)^s$ (see \cite{Paley}) and $(\zed_{p^{\frac{s}{2}}})^2$ when $s$ is even (see \cite{LM}).
\item $v=n^4$ or $v=9n^4$ where $n>1$ is odd. Suppose $n$ has prime factorization $n=\prod_{i=1}^r p_{i}^{t_i}$, then such partial difference sets exist in groups $G$, where $G = \prod_{i=1}^r (\zed_{p_i})^{4t_i}$ or $G = (\zed_3)^2 \times \prod_{i=1}^r (\zed_{p_i})^{4t_i}$ (see \cite{Polhill}).
\end{enumerate}
\item $D$ is a $(243,22,1,2)$-PDS. Such a partial difference set exists in $\zed_3^5$ (see \cite{BLS}).
\end{enumerate}
\end{theorem}

Infinite families of symmetric $\NF[(s,t)]{G}{\lambda}$ with $|G|=s+t+1$ follow from Corollary \ref{PDScharacterization.cor} and Theorem \ref{PDScharacterization.prop}.

\begin{corollary} 
\label{cor2.21}
Symmetric (s,t)-near-factorizations of a group $G$ with $|G|=s+t+1$ exist in the follow cases.
\begin{enumerate}
\item  Symmetric $\NF[(\frac{p^s-1}{2},\frac{p^s-1}{2})]{G}{\frac{p^s-1}{4}}$ exist when $G=(\zed_p)^s$ and $p$ is an odd prime with $p^s \equiv 1 \bmod 4$. 
\item  Symmetric $\NF[(\frac{p^{2s}-1}{2},\frac{p^{2s}-1}{2})]{G}{\frac{p^{2s}-1}{4}}$ exist when $G=(\zed_{p^s})^2$, $p$ is an odd prime and $s \ge 2$. 
\item  Symmetric $\NF[(\frac{n^4-1}{2},\frac{n^4-1}{2})]{G}{\frac{n^4-1}{4}}$ exist when $n>1$ is odd with prime factorization $n=\prod_{i=1}^r p_{i}^{t_i}$ and $G=\prod_{i=1}^r (\zed_{p_i})^{4t_i}$.
\item  Symmetric $\NF[(\frac{9n^4-1}{2},\frac{9n^4-1}{2})]{G}{\frac{9n^4-1}{4}}$ exist when $n>1$ is odd with prime factorization $n=\prod_{i=1}^r p_{i}^{t_i}$ and $G=(\zed_3)^2 \times \prod_{i=1}^r (\zed_{p_i})^{4t_i}$.
\item  A symmetric $\NF[(22,220)]{(\zed_3)^5}{20}$ exists.
\end{enumerate}
\end{corollary}

\subsection{A product construction}
\label{prod.sec}

An \NF[(s,t)]{G}{\lambda}  is defined to be a \emph{half-set near-factorization} if 
$t = (|G|-1)/2$. Of course $|G|$ must be odd for a half-set NF in $G$ to exist. Equation (\ref{prod.eq}) states that $st = \lambda (|G|-1)$, so an NF is a half-set NF if and only if $s = 2 \lambda$.

We will describe a product construction for half-set NFs. Suppose we have two 
half-set NFs. Specifically, suppose  
$(S_i,T_i)$ is a half-set \NF[(s_i,t_i)]{G_i}{\lambda_i},
for $i = 1,2$.  Denote $|G_i| = n_i$, $i = 1,2$.
Define the sets $S_3$ and $T_3$ as follows:
\begin{align}
\label{Peq1} S_3 &= S_1 \times S_2, \quad \text{and}\\
\label{Peq2} T_3 &= ((G_1 \setminus T_1) \times T_2) \; \bigcup \; (T_1 \times (G_2 \setminus T_2)).
 \end{align}
We will prove that $(S_3,T_3)$ is a half-set \NF[(s_1s_2,(n_1n_2-1)/2)]{G_1\times G_2}{2\lambda_1\lambda_2}.

\begin{example}
{\rm If $S_1 = \{1,2\}$ and  $T_1 = \{0\}$, then  $(S_1,T_1)$ is a half-set \NF[(1,2)]{\zed_3}{1}.
Also, if $S_2 = \{2,3\}$ and  $T_2 = \{1,4\}$, then  $(S_2,T_2)$ is a half-set \NF[(2,2)]{\zed_5}{1}.
If we construct $S_3$ and $T_3$ using the formulas (\ref{Peq1}) and (\ref{Peq2}), then we obtain
\begin{align*}
 S_3 &= \{ (1,2), (1,3), (2,2), (2,3)\}, \quad \text{and}\\
 T_3 &= \{ (1,1), (1,4), (2,1), (2,4)\} \; \bigcup \; \{(0, 0), (0, 2), (0,3) \}.
 \end{align*}
It is straightforward to verify that  $(S_3,T_3)$ is a \NF[(4,7)]{\zed_3 \times \zed_5}{2}.}$\hfill\blacksquare$
\end{example}

In general, (using group ring notation for convenience), we have the following:
\begin{align}
\label{Peq3}S_1  T_1 &= \lambda_1(G_1 - 0)\\
\label{Peq4}S_2  T_2 &= \lambda_2(G_2 - 0)\\
\label{Peq5}S_1  G_1 &= s_1 G_1\\
\label{Peq6}S_2  G_2 &= s_2 G_2\\
\nonumber S_1  (G_1 - T_1) &= s_1 G_1 - \lambda_1(G_1 - 0) \quad\quad \text{from (\ref{Peq3}) and (\ref{Peq5})}\\
\label{Peq7}&= (s_1-\lambda_1)G_1 + \lambda_1 0 \\
\nonumber S_2  (G_2- T_2) &= s_2 G_2 - \lambda_2(G_2 - 0) \quad\quad \text{from (\ref{Peq4}) and (\ref{Peq6})}\\
\label{Peq8}&= (s_2-\lambda_2)G_2 + \lambda_2 0.
\end{align}

We also have
\begin{align}
S_1  (G_1 - T_1) \times S_2  T_2 &= ((s_1-\lambda_1)G_1 + \lambda_1 0) \times \lambda_2 (G_2 - 0) 
\quad\quad \text{from (\ref{Peq4}) and (\ref{Peq7})}\\
S_1  T_1 \times S_2  (G_2- T_2)  &= \lambda_1 (G_1 - 0) \times ((s_2-\lambda_2)G_2 + \lambda_2 0)
\quad\quad \text{from (\ref{Peq3}) and (\ref{Peq8}).}\end{align} 

We want to compute $S_3T_3$. To do so, we use the identity 
\begin{equation}
\label{Peq11} (A \times B) (C \times D) = (AC) \times (BD).
\end{equation}

Then, using the facts that $s_1 = 2 \lambda_1$ and $s_2 = 2 \lambda_2$, we have
\begin{align*}
S_3  T_3 &= \big(S_1  (G_1 - T_1) \times S_2  T_2\big) + 
\big(S_1  T_1 \times S_2  (G_2- T_2)\big) \quad\quad \text{from (\ref{Peq1}), (\ref{Peq2}) and (\ref{Peq11})}\\
&= \big(((s_1-\lambda_1)G_1 + \lambda_1 0) \times \lambda_2 (G_2 - 0)\big) \\
&+ 
\big(\lambda_1 (G_1 - 0) \times ((s_2-\lambda_2)G_2 + \lambda_2 0)\big) 
\quad\quad \text{from (\ref{Peq3}), (\ref{Peq4}), (\ref{Peq7}) and (\ref{Peq8})}\\
&= (s_1-\lambda_1)\lambda_2(G_1 \times G_2) - (s_1-\lambda_1)\lambda_2(G_1 \times 0) 
+ \lambda_1\lambda_2(0 \times G_2) - \lambda_1\lambda_2(0 \times 0)\\
&+ \lambda_1(s_2-\lambda_2) (G_1 \times G_2) - \lambda_1(s_2-\lambda_2)(0 \times G_2)  
+ \lambda_1\lambda_2(G_1 \times 0) - \lambda_1\lambda_2(0 \times 0)\\
&= \lambda_1\lambda_2(G_1 \times G_2) - \lambda_1\lambda_2(G_1 \times 0) 
+ \lambda_1\lambda_2(0 \times G_2) - \lambda_1\lambda_2(0 \times 0)\\
&+ \lambda_1\lambda_2 (G_1 \times G_2) - \lambda_1\lambda_2(0 \times G_2)  
+ \lambda_1\lambda_2(G_1 \times 0) - \lambda_1\lambda_2(0 \times 0)\\
&= 2\lambda_1\lambda_2\big((G_1 \times G_2) - (0,0)\big).
\end{align*}
Thus $(S_3,T_3)$ is  an \NF[(s_3,t_3)]{G_1 \times G_2}{\lambda_3},
where 
\begin{align*}
\lambda_3 &= 2 \lambda_1\lambda_2\\
s_3 &= s_1s_2, \quad\quad \text{and}\\
t_3 &= t_2(n_1-t_1) + t_1(n_2 - t_2)\\
&= \left( \frac{n_2-1}{2}\right)\left( n_1 - \frac{n_1-1}{2}\right)
+ \left( \frac{n_1-1}{2}\right)\left( n_2 - \frac{n_2-1}{2}\right)\\
&= \left( \frac{n_2-1}{2}\right)\left(\frac{n_1+1}{2}\right)
+ \left( \frac{n_1-1}{2}\right)\left(\frac{n_2+1}{2}\right)\\
&= \frac{n_1n_2-1}{2}.
\end{align*}
Therefore $(S_3,T_3)$ is also a half-set near-factorization.

\medskip

We have proven the following theorem.

\begin{theorem}
\label{prod-thm}
Suppose there exists a half-set \NF[(s_i,t_i)]{G_i}{\lambda_i}, for $i = 1,2$. Denote $|G_i| = n_i$, $i = 1,2$. 
Then there exists a half-set \NF[(s,t)]{G_1\times G_2}{2\lambda_1\lambda_2}, where $s = s_1s_2$, $\lambda = 2\lambda_1\lambda_2$ and 
$t = (n_1n_2-1)/2$.
\end{theorem}

An immediate consequence of Theorem \ref{prod-thm} is that we can take products of an arbitrary number half-set NFs and the result is again a half-set NF.
It is also true that the product of symmetric half-set NFs is symmetric.

\begin{corollary}
\label{prod-thm-sym}
Suppose there exists a symmetric half-set \NF[(s_i,t_i)]{G_i}{\lambda_i}, for $i = 1,2$. Denote $|G_i| = n_i$, $i = 1,2$. 
Then there exists a symmetric half-set \NF[(s,t)]{G_1\times G_2}{2\lambda_1\lambda_2}, where $s = s_1s_2$, $\lambda = 2\lambda_1\lambda_2$ and 
$t = (n_1n_2-1)/2$.
\end{corollary}

\begin{proof} Suppose $(S_i,T_i)$ is a  symmetric half-set \NF[(s_i,t_i)]{G_i}{\lambda_i}, for $i = 1,2$. If $S_1$ and $S_2$ are symmetric, then it is clear that $S_3 = S_1 \times S_2$ is also symmetric. From Theorem \ref{prod-thm}, we have that $(S_3,T_3)$ is an \NF[(s,t)]{G_1\times G_2}{\lambda}, where $T_3$ is obtained from  (\ref{Peq2}), $s = s_1s_2$, $\lambda = 2\lambda_1\lambda_2$ and 
$t = (n_1n_2-1)/2$. Because $S_3$ is symmetric, it follows from Lemma \ref{SimpliesT.lem} that $T_3$ is also symmetric.
\end{proof}

\medskip

Now we describe some applications of the product construction.
It was proven by de Bruijn \cite{DB} that there is an \NF[(s,t)]{\zed_n}{1} whenever $n-1 = st$. 
When $n$ is odd and $s=2$, we get a half-set \NF[\left(2,\frac{n-1}{2}\right)]{\zed_n}{1}. (In fact, it is easy to write down a near-factorization with these parameters, namely, $S = \{0,1\}$ and $T = \{1,3,5, \dots, n-2\}$.) 
This is a near-factorization with $\lambda = 1$, so it is equivalent to a symmetric 
\NF[\left(2,\frac{n-1}{2}\right)]{\zed_n}{1} by Lemma \ref{symmetric.lem}. Applying Corollary \ref{prod-thm-sym}, we obtain the following result by induction on $k$.

\begin{theorem}
\label{general.thm}
Suppose $k \geq 1$ is an integer and suppose $n_j$ is a positive odd integer exceeding 1, for $1 \leq j \leq k$. Then there exists a symmetric half-set \NF[\left(2^k,\frac{(\prod_{j=1}^{k} n_j) -1}{2}\right)]{\zed_{n_1} \times \cdots \times \zed_{n_{k}}}{2^{k-1}}.
\end{theorem}

Suppose $G$ is an abelian group.
Define $r(G)$ (the \emph{rank} of $G$) to be the minimum number of cyclic groups in any decomposition of $G$ into a 
direct product of cyclic groups. Define $m(G)$ to be the number of cyclic groups in the unique decomposition of $G$ into a direct product of cyclic groups of prime power order. Clearly we have $r(G) \leq m(G)$. As an example, if 
$G = \zed_9 \times \zed_{15} \times \zed_{25}$, then $r(G) = 2$ (via the decomposition $G = \zed_{15} \times \zed_{2 25}$) and  $m(G) = 4$ (via the decomposition $G = \zed_9 \times \zed_{3} \times \zed_{5} \times \zed_{25}$).

We have the following corollary of Theorem \ref{general.thm}.

\begin{corollary}
Let $G$ be an abelian group of odd order $n$ and suppose $r(G) \leq k \leq m(G)$. Then there exists a
symmetric half-set \NF[\left(2^k,\frac{n -1}{2}\right)]{G}{2^{k-1}}.
\end{corollary}

\begin{proof}
For the stated values of $k$, it is possible to express $G$ as a direct product of $k$ cyclic groups, each having odd order exceeding 1. Apply Theorem \ref{general.thm}.
\end{proof}

\subsubsection{Applications of the product construction using difference sets and partial difference sets}

We now  show how difference sets and partial difference sets give rise to half-set near-factorizations and we discuss applications using the product construction.

A difference set with parameters $(4n-1,2n-1,n-1)$ is called a \emph{Paley-type} difference set.
It follows from  from Theorem \ref{differencesets.thm}  that a Paley-type $(4n-1,2n-1,n-1)$-difference set in an abelian group $G$ of order $4n-1$ gives rise to a \NF[(2n-1,2n)]{G}{n}. This is equivalent to a \NF[(2n,2n-1)]{G}{n}, which is a half-set near-factorization. Paley-type difference sets are known to exist in the following cases:
\begin{itemize}
\item When $q = 4n-1$ is a prime power, the quadratic residues in $\eff_q$ form a Paley-type $(q,(q-1)/2,(q-3)/4)$-difference set.
\item Suppose $q$ and $q+2$ are both prime powers. Then there is a Paley-type difference set in 
$\eff_q \times \eff_{q+2}$ (see \cite{StSp}).
\item The Singer difference sets in $\zed_{2^{d+1}-1}$ are Paley-type $(2^{d+1}-1, 2^{d}-1, 2^{d-1}-1)$-difference sets for all integers $d \geq 2$.
\end{itemize}

\begin{example} 
\label{ex21}
{\rm There are Paley-type $(3,1,0)$- and $(7,3,1)$-difference sets. These yield
a half-set \NF[(2,1)]{\zed_3}{1} and a half-set \NF[(4,3)]{\zed_7}{2}, respectively. 
The product of these two NFs is a half-set \NF[(8,10)]{\zed_3 \times \zed_7}{4}. 

Here are the details:
\begin{itemize}
\item  $\{0\}$ is a (trivial) $(3,1,0)$-difference set, which gives rise to a \NF[(2,1)]{\zed_3}{1}, namely
$S_1 = \{1,2\}$, $T_1 = \{0\}$ in $\zed_3$.
\item  $\{1,2,4\}$ is a  $(7,3,1)$-difference set, which gives rise to a \NF[(4,3)]{\zed_7}{2}, namely
$S_2 = \{0,3,5,6\}$, $T_2 = \{1,2,4\}$ in $\zed_7$.
\item The product of these two near factorizations is
\begin{align*}
S_3 &= \{(1,0), (1,3), (1,5), (1,6), (2,0), (2,3), (2,5), (2,6)\}\\
T_3 &= \{ (0,0), (0,3), (0,5), (0,6), (1,1), (1,2), (2,1), (2,2), (4,1), (4,2)\}.
\end{align*}
\end{itemize} 
This near-factorization is not symmetric because the \NF[(4,3)]{\zed_7}{2} is not symmetric.}
$\hfill\blacksquare$
\end{example}

\begin{example}
\label{ex33} 
{\rm There are Paley-type $(3,1,0)$- and $(11,5,2)$-difference sets. These yield
a half-set \NF[(2,1)]{\zed_3}{1} and a half-set \NF[(6,5)]{\zed_{11}}{3}, respectively. 
The product of these two NFs is a (nonsymmetric) half-set \NF[(12,16)]{\zed_3 \times \zed_{11}}{6}. }
$\hfill\blacksquare$
\end{example}

We can also obtain half-set NFs from certain partial difference sets. A \emph{Paley-type} partial difference set (PDS) is a PDS with parameters $(4n+1,2n,n-1,n)$ (these were discussed in Theorem \ref{PDScharacterization.prop}).  It follows from   Theorem \ref{PDStoNF.thm}  that a Paley-type $(4n+1,2n,n-1,n)$-PDS in an abelian group $G$ of order $4n+1$ gives rise to a half-set \NF[(2n,2n)]{G}{n}.


We have described three families of half-set NFs. We can of course  ``mix-and-match'' NFs of various types in applications of the product construction.

\begin{example} 
\label{ex35}
{\rm There is a Paley-type $(7,3,1)$-difference set  and a Paley-type $(5,2,1,1)$-PDS. These yield
a half-set \NF[(4,3)]{\zed_7}{2} and a non-symmetric half-set \NF[(2,2)]{\zed_{5}}{1}, respectively. 
The product of these two NFs is a non-symmetric half-set \NF[(8,17)]{\zed_5 \times \zed_{7}}{4}.}
$\hfill\blacksquare$
\end{example}

\begin{example}
\label{ex39} 
{\rm There is a Paley-type $(3,1,0)$-difference set  and a Paley-type $(13,6,2,3)$-PDS. These yield
a symmetric half-set \NF[(2,1)]{\zed_3}{1} and a symmetric half-set \NF[(6,6)]{\zed_{13}}{3}, respectively. 
The product of these two NFs is a symmetric half-set \NF[(12,19)]{\zed_3 \times \zed_{13}}{6}.}
$\hfill\blacksquare$
\end{example}

\section{Restrictions on the sizes of subsets in a near-factorization}
\label{subsetsize.sec}

Suppose $G$ and $H$ are groups. If there is a group homomorphism $\tau : G \rightarrow H$, then we can extend 
$\tau$ to a group ring homomorphism 
$\tau : \zed [G] \rightarrow \zed [H]$ in an obvious way:
\[ \tau \left( \sum_{g \in G} a_g \, g \right) = \sum_{g \in G} a_g \, \tau(g).\]

Let $K = \{g \in G : \tau(g)= e \}$ be the kernel of a surjective group homomorphism $\tau$. Let $k=|K|$ and $h=|H|$; observe that
$k = |G| /h$.
It follows easily that
\[ \tau(G) = kH.\]
Then, in the group ring $\zed [G]$, we have
\begin{equation}
\label{Sec2.eqn}
\tau\Bigl(\lambda (G - e) \Bigr)
= \lambda \Bigl(\tau (G - e) \Bigr)
= \lambda \Bigl(\tau (G) - e \Bigr)
= \lambda \Bigl(kH - e \Bigr)
.\end{equation}

The extension of the constant homomorphism
$\delta \colon G \rightarrow \{e\}$ to $\zed [G]$
is called the \emph{augmentation map}:
\[
\delta\Bigl(\sum_{g \in G}x_g \, g\Bigr) = \sum_{g \in G}x_g .
\]
Observe that, if $S \subseteq G$, then $\delta(S) = |S|$.

We now fill in a gap in the proof of \cite[Theorem 3]{CGHK}. 

\begin{theorem}
\label{t1}
If the (not necessarily abelian) group $G$ has the additive abelian group $(\zed_p)^m$ as a quotient group, then for any near-factorization $G -e = ST$, we have
\begin{align}
|S|^{p-1} \equiv |T|^{p-1} \equiv 1 \bmod {p^m} & \quad \text{if $p$ is an odd prime}\label{eq1}\\
|S| \equiv -|T| \equiv \pm 1 \bmod {2^m} &  \quad \text{if $p=2$.}\label{eq2}
\end{align}
\end{theorem}

\begin{proof}
We follow the proof of  \cite[Theorem 3]{CGHK} with only minor changes of notation.
For convenience, let us denote $s = |S|$, $t = |T|$, $q = p^m$ and $k = |G| / q$.
Let $\tau : G \rightarrow H = (\zed_p)^m$ be a surjective homomorphism, where $p$ is prime. We will work in the group ring $\zed [H]$. 

Let $A = \tau(S)$, $B = \tau(T)$, where $A,B \in \zed [H]$. It follows immediately from (\ref{Sec2.eqn}) (with $\lambda = 1$) that 
\begin{equation}
\label{eq3} AB = kH - e.
\end{equation}

We can identify $H$ with the additive group of the finite field $\eff_q$. Let $\alpha$ be a primitive element of $\eff_q$. The mapping
$\sigma : H \rightarrow H$ defined as $\sigma(x) = \alpha x$ is an (additive) homomorphism of $H$ having order $q-1$. For $0 \leq i \leq q-2$, define
$A_i = \sigma^i (A) = \alpha^i A$ and $B_i = \sigma^i (B) = \alpha^i B$.

\begin{equation}
\label{eq4} A_iB_i = kH - e
\end{equation}
for $i = 0, \dots , q-2$.
Define 
\[ A' = \prod_{i=0}^{q-2} A_i \quad \text{and} \quad B' = \prod_{i=0}^{q-2} B_i.\]

Multiplying the $q-1$ equations (\ref{eq4}) together, we have
\begin{equation}
\label{eq5}
 A'B' = (kH-e)^{q-1}.
\end{equation}
It is clear that $H^2 = q H$. Hence, if we expand $(kH-e)^{q-1}$ using the binomial theorem, it follows from (\ref{eq5}) that
\begin{equation}
 A'B' = 
 \begin{cases}
 rH + e & \text{if $q$ is odd}\\
 rH - e & \text{if $q$ is even,}
 \end{cases}
\label{eq6}
\end{equation}
where $r$ is an integer.

We can write
\[ A' = \sum_{x \in H} c_x x,\]
where the $c_x$'s are integer coefficients. We have $\sigma(A') = A'$ and 
 $\sigma$ is transitive on $H - e$. Hence, it follows that all the $c_x$'s are equal, for $x \neq e$. 
 A similar result holds for $B'$. Hence, we have 
\begin{equation}
\label{eq7}
A' = ae + bH \quad \text{and} \quad B' = ce + dH,\end{equation}
where $a,b,c,d$ are integers.
Substituting (\ref{eq7}) into (\ref{eq6}), we see that
$ac = \pm 1$, so $a^2 = c^2 = 1$. 

We also have
\begin{equation}
\label{eq8}
 s^{q-1} = \delta(A)^{q-1} = \delta(A') = a + bq.
 \end{equation}
Because $a^2 = 1$, we see from (\ref{eq8}) that
\begin{equation}
\label{eq9} s^{2(q-1)} \equiv 1 \bmod q.
\end{equation}
It is also true from Euler's theorem that 
\[
 s^{\phi(q)} \equiv 1 \bmod q.
 \]
 Because $\phi(q) = p^{m-1}(p-1)$, we obtain
 \begin{equation}
\label{eq10} s^{p^{m-1}(p-1)} \equiv 1 \bmod q.
\end{equation}
We have 
\begin{equation}
\label{eq11}
\gcd(2(p^m-1), p^{m-1}(p-1)) =
\begin{cases}
p-1 & \text{if $p$ is odd or $m = 1$}\\
2 & \text{if $p=2$ and $m \geq 2$.}\\
\end{cases}
\end{equation}
From (\ref{eq9}), (\ref{eq10}) and (\ref{eq11}), we see that
\begin{equation}
\label{eq12}
s^{p-1} \equiv 1 \bmod p^m
\end{equation}
if $p$ is an odd prime
and
\begin{equation}
\label{eq13}
s^{2} \equiv 1 \bmod 2^m
\end{equation}
if $p = 2$.
The same congruences hold for $t = |T|$.

If $p$ is an odd prime, then we have proven (\ref{eq1}). (So far, we have followed the proof of \cite[Theorem 3]{CGHK}.)

If $p=2$ and $m=1$, then $s \equiv t  \equiv 1 \bmod 2$ and (\ref{eq2}) holds.

If $p=2$ and $m \ge 2$, we note that
\[ st = |G| - e = kq-1 \equiv -1 \bmod q.\]
Consequently, we obtain the following system of congruences:
\begin{equation}
\label{eq14}
\begin{aligned}
s^{2} &\equiv 1 \bmod 2^m\\
t^{2} &\equiv 1 \bmod 2^m\\
st &\equiv -1 \bmod 2^m.
\end{aligned}
\end{equation}

If $p=2$ and $m=2$, then $s \equiv -t  \equiv \pm 1 \bmod 4$ and (\ref{eq2}) holds.

The case  $p=2$ and $m \ge 3$ remains to be proven. It is stated in the last line of the proof of \cite[Theorem 3]{CGHK} that (\ref{eq2}) follows as a consequence
of the three congruences in (\ref{eq14}). The first and second equations in (\ref{eq14}) each has  four solutions modulo $q$ (where $q = 2^m$), namely,  $s, t \in \{ 1,2^{m-1}-1, 2^{m-1}+1, 2^m-1 \}$. The solution to the system of three congruences given in  (\ref{eq14}) is 
\begin{equation}
\label{eq15}
(s,t) \in \{ (1,2^m-1), (2^m-1,1), (2^{m-1}-1,2^{m-1}+1),  (2^{m-1}+1,2^{m-1}-1) \},
\end{equation}
which includes the solution in (\ref{eq2}) as a proper subset.
For example, if we take $m = 4$ (so $q = 16$), then $s = 7$, $t = 9$ is a solution to (\ref{eq15}) in which (\ref{eq2}) is not satisfied. 

\medskip

Fortunately, we can prove that (\ref{eq2}) holds with a small amount of additional work. First, from (\ref{eq5}), (\ref{eq6}) and (\ref{eq7}), we have
$ac = -1$ because $q$ is even. Hence, $a = -c = \pm 1$.
From (\ref{eq8}), we have
\begin{equation}
\label{eq16}
s^{2^m-1} \equiv a \equiv \pm 1 \bmod 2^m.
\end{equation}
On the other hand, 
\[
(2^{m-1} \pm 1)^{2^m-1} = \sum_{i=0}^{2^m-1} \binom{2^m-1}{i} (2^{m-1})^{i} (\pm 1)^{2^m-1-i}.
\] 
All the terms in this binomial expansion are divisible by $2^m$ except for the terms corresponding to $i=0$ and $i=1$.
The term corresponding to $i = 0$ is $\pm 1$ and the term corresponding to $i = 1$ is 
\[
 (2^m-1)2^{m-1} \equiv  2^{m-1} \bmod 2^m.\]
 It follows that 
 \[ (2^{m-1} \pm 1)^{2^m-1} \equiv 2^{m-1} \pm 1 \bmod 2^m.\]
 Therefore, from (\ref{eq16}), we must have 
\begin{equation*}
(s,t) \in \{ (1,2^m-1), (2^m-1,1) \},
\end{equation*}
and hence (\ref{eq2}) is satisfied. This completes the proof of Theorem \ref{t1}.
\end{proof}

\subsection{Generalization to $\lambda > 1$}

When $\lambda > 1$, we can extend Theorem  \ref{t1} provided that 
$\ell \equiv \pm 1 \bmod q$ for every prime divisor $\ell$ of $\lambda$. As before, we assume that $q = p^m$ for some prime $p$ and 
the group $G$ has the additive abelian group $(\zed_p)^m$ as a quotient group. Suppose that $(S,T)$ is 
an $\NF[(s,t)]{G}{\lambda}$.
We summarize the main changes in the proof. 
First, equation (\ref{eq5}) becomes
\begin{equation}
\label{eq18}
 A'B' = \lambda^{q-1}(kH-e)^{q-1}
\end{equation}
and equation (\ref{eq6}) is now 
\begin{equation}
 A'B' = 
 \begin{cases}
 \lambda^{q-1}(rH + e) & \text{if $q$ is odd}\\
 \lambda^{q-1}(rH - e) & \text{if $q$ is even,}
 \end{cases}
\label{eq19}
\end{equation}
where $r$ is an integer.

Equation (\ref{eq7}) holds (for different values of $a,b,c$ and $d$ than before, of course). Then we see that 
$ac =  \lambda^{q-1}$ if $q$ is odd and $ac = - \lambda^{q-1}$ if $q$ is even.
Because $\ell \equiv \pm 1 \bmod q$ for every prime divisor $\ell$ of $\lambda$, 
it follows that every prime divisor of $a$ or $c$ is congruent to $\pm 1$ modulo $q$. 
Hence,
\begin{equation}
\label{eq21}
a^2 \equiv 1 \bmod q \quad \text{and} \quad c^2 \equiv 1 \bmod q.
\end{equation} 
Equation (\ref{eq8}) can be derived as in the case $\lambda = 1$. Then we can use (\ref{eq21}) to show that (\ref{eq9}) is still true.
 If $q$ is odd, then we obtain (\ref{eq1}) as in the case $\lambda = 1$.
Therefore we now consider the case $p=2$.

 Because $\ell \equiv \pm 1 \bmod q$ for every prime divisor $\ell$ of $\lambda$, it follows that  $\lambda$ must be  odd. Because $q \mid n$, we have that $n$ is even. Then, because $st=\lambda(n-1)$, we see that  $s$ and $t$ are odd. Therefore, 
 for $p=2$ and $m=1$, we have $s \equiv t \equiv \lambda \equiv 1 \bmod 2$ and hence (\ref{eq2}) holds.

Now assume $p=2$ and $m \ge 2$.  Because (\ref{eq13}) still holds, we have
\begin{equation}
\label{eq22}
 s^{2} \equiv t^{2}  \equiv 1 \bmod q.
\end{equation}

Also, from (\ref{eq8}), we have
\begin{equation}
\label{eq23}
 s^{2^{m}-1} \equiv a \bmod q \quad \text{and} \quad t^{2^{m}-1} \equiv c \bmod q.
\end{equation}
It follows that
\[as \equiv s^{2^{m}} \equiv (s^2)^{2^{m-1}} \equiv 1 \equiv s^2 \bmod q.\]
Similarly, $ct \equiv t^2 \bmod q$. Hence,
\begin{equation}
\label{eq24}
s \equiv a \bmod q \quad \text{and} \quad t \equiv c \bmod q.
\end{equation}

We are assuming that $\ell \equiv \pm 1 \bmod q$ for every prime divisor $\ell$ of $\lambda$.
We  have $ac = - \lambda ^{q-1}$, so it follows that $a \equiv \pm 1 \bmod q$ and $c \equiv \pm 1 \bmod q$. 

We also have $\lambda ^{q-1} \equiv \pm 1 \bmod q$. When $\lambda^{q-1} \equiv 1 \bmod q$, we have $ac \equiv -1 \bmod q$ and hence $a \equiv -c \equiv \pm 1 \bmod q$.
When $\lambda^{q-1} \equiv -1 \bmod q$, we have $ac \equiv 1 \bmod q$ and hence $a \equiv c \equiv \pm 1 \bmod q$.
Together with (\ref{eq24}), we have the following generalization of Theorem \ref{t1}.

\begin{theorem}
\label{t2}
Suppose that 
$\ell \equiv \pm 1 \bmod q$ for every prime divisor $\ell$ of $\lambda$, where $q = p^m$ for some prime $p$.
If the (not necessarily abelian) group $G$ has the additive abelian group $(\zed_p)^m$ as a quotient group, then for any $\lambda$-fold near-factorization $\lambda(G - e) = ST$, we have
\begin{align*}
|S|^{p-1} \equiv |T|^{p-1} \equiv 1 \bmod {p^m} & \quad \text{if $p$ is an odd prime}\\
|S| \equiv -|T| \equiv \pm 1 \bmod {2^m} &  \quad \text{if $p=2$ and $\lambda^{q-1} \equiv  1 \bmod q$}\\
|S| \equiv |T| \equiv \pm 1 \bmod {2^m} &  \quad \text{if $p=2$ and $\lambda^{q-1} \equiv - 1 \bmod q$.}
\end{align*}
\end{theorem}

We now present a corollary to Theorem \ref{t2}, which is a generalization of \cite[Example 5]{CGHK}.



\begin{corollary}
\label{cong.ex}
Suppose  $p$ is an odd prime and $\ell \equiv \pm 1 \bmod p^2$ for every prime divisor $\ell$ of $\lambda$. Then no 
$\NF[\bigl( p-1,\lambda(p+1) \bigr) ]{(\zed_p)^2}{\lambda}$ or
$\NF[\bigl( p+1, \lambda(p-1) \bigr) ]{(\zed_p)^2}{\lambda}$ exists.
\end{corollary}

\begin{proof}
Let $H =G = (\zed_p)^2$ and apply the identity mapping from $G$ to itself in Theorem \ref{t2}.
First, let $s=p-1$ and $t=\lambda(p+1)$. Then 
we have 
\[
s^{p-1}=(p-1)^{(p-1)} \equiv p+1 \not\equiv 1 \bmod {p^2}
\]
from the binomial theorem. 

Also, if instead we take $s=p+1$, $t=\lambda(p-1)$, then 
 we have 
\begin{center}
{$s^{p-1}=(p+1)^{(p-1)} \equiv -p+1 \not\equiv 1 \bmod {p^2},$}
\end{center}
by the binomial theorem. 
\end{proof}

\section{Computational results}\label{Comp.sec}

In this section, we present our algorithm for computing $\lambda$-mates and we use a computer implementation  to construct near-factorizations of abelian groups of orders at most $50$.
Our algorithm is a straightforward generalization of the techniques developed in \cite[\S 1.1]{KMS} in the case $\lambda = 1$. We recall the notation from Section \ref{properties.sec}.

Let $G=\{g_1,g_2,\ldots,g_n\}$ be a multiplicative group with identity $e=g_1$ 
 and suppose $B \subseteq G$.
The $(0,1)$-valued matrix $M(B)$ is defined as follows:
\begin{align*}
 (M(B))_{i,j} &=1 \text{ if and only if }  g_i^{-1}g_j  \in B.
 \end{align*}
The following theorem is immediate.

\begin{theorem}\label{FirstColumn.thm} 
Let $G=\{g_1,g_2,\ldots,g_n\}$ be a \textup{(}multiplicative\textup{)} group with identity $e=g_1$. 
Suppose $B \subseteq G$ and define
\[ x_i = 
\begin{cases}
1~\text{ if $g_i^{-1} \in B$}\\
0~\text{ otherwise.}
\end{cases}
\]
Then 
\[
 (M(B))_{i,j} = x_k,
\]
where $g_k=g_j^{-1}g_i$.
\end{theorem}

\begin{proof}
Because $g_k=g_j^{-1}g_i$, we have  $g_k^{-1}=g_i^{-1}g_j$.
Then,
\[ (M_B)_{i,j} = 1 \Leftrightarrow g_i^{-1}g_j \in B \Leftrightarrow g_k^{-1} \in B  \Leftrightarrow  (M_B)_{k,1} = 1.
\]
It then follows that $(M(B))_{i,j} = (M(B))_{k,1} = x_k$. \qedhere
\end{proof}

According to Theorem~\ref{FirstColumn.thm}, we can define a vector $X=[x_1,x_2,\ldots,x_n]^T$, 
which is the first column of $M(B)$. Therefore, $M(B)$ is completely determined by the entries in the first column of $M(B)$.

If $T$ is a $\lambda$-mate to the subset $S \subseteq G$,
then by~Theorem \ref{equiv.thm},  $M(S)M(T) = \lambda (J-I)$. 
Hence, because $\lambda (J-I)$ has nonzero  determinant, the equation
\begin{equation}
 M(S) X = [0,\underbrace{\lambda,\lambda \cdots, \lambda}_{n-1~\text{times}}]^T\label{first.eqn}
\end{equation}
has a unique rational solution $X$. If $X$ is $(0,1)$-valued, then, applying
Theorem~\ref{FirstColumn.thm}, a subset $T$ is obtained that is a $\lambda$-mate to $S$.
Furthermore, this subset $T$ is unique. Consequently, an algorithm to compute the $\lambda$-mate
for a given subset $S$ of finite group $G$ with $n=|G|$, if exists, is as follows:
\begin{enumerate}
\item Compute  $M(S)$.
\item Solve $M(S) X = [0,\underbrace{\lambda,\lambda \cdots, \lambda}_{n-1~\text{times}}]^T$.
\item If $X=(x_1,x_2,\ldots,x_n)$ is $(0,1)$-valued, then 
\[
T= \{g_i^{-1} : x_i\neq 0\}
\]
is the unique $\lambda$-mate $T$ to $S$; otherwise $S$ has no $\lambda$-mate.
\end{enumerate}

This algorithmic approach  can be improved by taking advantage of Theorem~\ref{T3-l.thm}.
First we note that equation~(\ref{first.eqn}) has a $(0,1)$-valued solution X if and only
if 
\begin{equation}
 M(S) Y = [0,\underbrace{1,1 \cdots, 1}_{n-1~\text{times}}]^T\label{second.eqn}
\end{equation}
has a $\left(0,\tfrac{1}{\lambda}\right)$-valued solution $Y$ in the rational field $\que$.
On the other hand, to significantly accelerate the computation, instead of  the rational field $\que$, we can solve equation~(\ref{second.eqn}) modulo $p$ for a large enough prime $p$.

Let $p$ be a prime with $p > s$. Then by Corollary~\ref{upperbound.cor} and Lemma~\ref{11.lem}, we have $p > \lambda$.
Equation~(\ref{second.eqn}) has  $\left(0,\tfrac{1}{\lambda}\right)$-valued solution $Y=[y_1,y_2,\ldots,y_n]^T$ in
$\que$ if and only if equation~(\ref{second.eqn}) has a
$\left(0,\lambda^{-1}\right)$-valued solution $\overline{Y}=[\bar{y}_1,\bar{y}_2,\ldots,\bar{y}_n]^T$ in $\zed_p$, where $\lambda^{-1}$ is the multiplicative inverse of $\lambda$ in $\zed_p$ and
$\bar{y_i} \in \zed_p$ satisfies $\bar{y_i} \equiv y_i \pmod {p}$ for each $1 \le i \le n$.

Conversely, if equation~(\ref{second.eqn}) has a  $(0,\lambda^{-1})$-valued solution $\overline{Y}=[\bar{y}_1,\bar{y}_2,\ldots,\bar{y}_n]^T$ in $\zed_p$, then 
\begin{equation}
M(S)\overline{Y} \equiv 
[0,\underbrace{1,1 \cdots, 1}_{n-1~\text{times}}]^T \pmod{p} \label{third.eqn}
\end{equation}
We claim that $Y=[y_1,y_2,\ldots,y_n]^T$ is a $(0,\frac{1}{\lambda})$-solution to equation~(\ref{second.eqn}) in $\que$, where $y_i=\frac{1}{\lambda}$ if and only if $\bar{y}_i=\lambda^{-1}$.  
The first row in equation~(\ref{third.eqn}) indicates
$$
\sum_{j=1}^n M(S)_{1,j} \bar{y}_j \equiv 0 \pmod{p}.
$$
Note that 
$$
0 \leq \sum_{j=1}^n M(S)_{1,j} y_j \leq \frac{s}{\lambda} < \frac{p}{\lambda} < p,
$$
which implies
$$
\sum_{j=1}^n M(S)_{1,j} y_j = 0~\text{in $\que$}.
$$
The remaining $n-1$ rows in equation~(\ref{third.eqn}) indicate that for each $2 \leq i \leq n$,
$$
\sum_{j=1}^n M(S)_{i,j} \bar{y}_j \equiv 1 \pmod{p}.
$$
Note that
$$
0 \leq \sum_{j=1}^n M(S)_{i,j} y_j \leq \frac{s}{\lambda} < \frac{p}{\lambda} < p,
$$
which implies
$$
\sum_{j=1}^n M(S)_{i,j} y_j = 1~\text{in $\que$}.
$$
Therefore, $Y=[y_1,y_2,\ldots,y_n]^T$ is a $(0,\frac{1}{\lambda})$-solution to equation~(\ref{second.eqn}) in $\que$.
Consequently, we have established the following lemma.
\begin{lemma}\label{mod p.lem}
Let $S$ be a subset of a finite group $G$ with $n=|G|$ and $s=|S|$. Let $p$ be a prime with $p>s$ Then $S$ has a $\lambda$-mate
$T$ in $G$ for some $\lambda$ if and only if there exists a $(0,\mu)$-solution $\overline{Y}$
$$
M(S)\overline{Y} \equiv 
[0,\underbrace{1,1 \cdots, 1}_{n-1~\text{times}}]^T \pmod{p}
$$ 
for some $\mu \in \zed_p$. Furthermore, $\lambda$ is the unique multiplicative inverse of $\mu$ in $\zed_p$.
\end{lemma}

Given a subset $S$ of a finite group $G$ with $s=|S|$ and a prime $p$ with $p>s$, a refined algorithm to compute the $\lambda$-mate
for a given subset $S$ of finite group $G$ with $n=|G|$ (if it exists) is as follows:

\begin{enumerate}
\item Compute  $M(S)$.
\item Solve $M(S)\overline{Y} \equiv [0,\underbrace{1,1 \cdots, 1}_{n-1~\text{times}}]^T \pmod{p}$.
\item If $\overline{Y}=(\bar{y}_1,\bar{y}_2,\ldots,\bar{y}_n)$ is a $(0,\mu)$-valued solution for some $\mu \in \zed_{p}$, then
\[
T= \{g_i^{-1} : \bar{y}_i \not= 0\}
\]
is the unique $\lambda$-mate $T$ to $S$, where $\lambda \equiv \mu^{-1} \pmod{p}$; otherwise $S$ has no $\lambda$-mate.
\end{enumerate}


The remaining difficulty in the computation is that, given possible parameters $s$, $t$, $\lambda$ and $G$, there are $\binom{|G|}{s}$ subsets $S$ to be tested, which leads to a huge search space. To address this difficulty, the upper bounds on $\lambda$ that we have proven in this paper greatly reduce the  number of parameter sets that need to be processed. Further, the list of candidate subsets $S$ can be abbreviated to only those that are least in lexicographical order with respect to equivalence. If only symmetric solutions are desired, then the search running time is dramatically reduced, because only subsets $S$ where $S=S^{(-1)}$ need to be considered.

The abelian groups in  Tables~\ref{Mixed.tab}, ~\ref{Pure.tab} and~\ref{Computer.tab}
are written additively
and classified using canonical decomposition. Thus, in these
tables, an abelian group $G$ of order $n$ is  written as
\[
G \cong 
\zed_{m_k}
\times \zed_{m_{k-1}}
\times \zed_{m_{k-2}}
\cdots
\times \zed_{m_{1}},
\]
for some $m_1,m_2,\ldots,m_k$, where $m_{i}|m_{i+1}$, $i=1,2,\ldots,k-1$, and $n= m_1m_2\ldots m_k$.
Note that the rank $k$ and the components $\zed_{m_i}$ in this decomposition are uniquely determined.

\begin{center}
\renewcommand{\arraystretch}{1.2}
\begin{longtable}{@{}r|l|r|r|r|c|@{\hspace{2pt}}R{8.75cm}l@{}}
\caption{ 
When a $\lambda$-fold $(s,t)$-near-factorization  of an abelian group of order $n \leq 35$ exists with $\lambda \geq 2$, either
an example is presented or an authority is given.
Parameters do not appear in the table if a
near-factorization with those parameters does not exist. If a symmetric solution can be found, then it is indicated in column \emph{Sym.?}.
}\\ \hline
\label{Mixed.tab}
$n$ & group & $s$ & $t$ & $\lambda$  &Sym.? & Authority  \\ \hline
\endfirsthead
\caption[]{(continued)}\\ \hline
$n$ & group & $s$ & $t$ & $\lambda$  &Sym.? & Authority  \\ \hline
\endhead
\hline
\endfoot
 7&$\zed_{7}$& 3& 4& 2&no&Theorem~\ref{differencesets.thm}, $D=\{$0, 1, 3$\}$\\\hline
 9&$(\zed_{3})^{2}$& 4& 4& 2&yes&Theorem~\ref{general.thm}\\\hline
11&$\zed_{11}$& 5& 6& 3&no&Theorem~\ref{differencesets.thm}, $D=\{$0, 1, 2, 4, 7$\}$\\\hline
13&$\zed_{13}$& 4& 9& 3&no&Theorem~\ref{differencesets.thm}, $D=\{$0, 1, 3, 9$\}$\\\hline
13&$\zed_{13}$& 6& 6& 3&yes&Theorem~\ref{PDStoNF.thm}, $D=\{$1, 3, 4, 9, 10, 12$\}$\\\hline
15&$\zed_{15}$& 4& 7& 2&yes&Theorem~\ref{general.thm} ($\zed_{15} \cong \zed_{5} \times \zed_3$)
\\\hline
15&$\zed_{15}$& 7& 8& 4&no&Theorem~\ref{differencesets.thm}, $D=\{$0, 1, 2, 4, 5, 8, 10$\}$\\\hline
16&$(\zed_{2})^{4}$& 6&10& 4&yes&Theorem~\ref{differencesets.thm}, $D=\{$(0,0,0,0), (0,0,0,1), (0,0,1,0), (0,1,0,0), (1,0,0,0), (1,1,1,1)$\}$\\\hline16&$(\zed_{4})^{2}$& 6&10& 4&yes&Theorem~\ref{differencesets.thm}, $D=\{$(0,1), (0,3), (1,0), (1,1), (3,0), (3,3)$\}$\\\hline
16&$\zed_{4}{\times}(\zed_{2})^{2}$& 6&10& 4&yes&Theorem~\ref{differencesets.thm}, $D=\{$(0,0,0), (0,0,1), (0,1,0), (1,0,0), (2,1,1), (3,0,0)$\}$\\\hline
16&$\zed_{8}{\times}\zed_{2}$& 5& 9& 3&no&$S=\{$(0,0), (0,1), (1,0), (3,0), (4,0)$\}$, $T=\{$(1,0), (1,1), (2,0), (3,0), (3,1), (4,1), (5,1), (6,0), (7,1)$\}$\\\hline
16&$\zed_{8}{\times}\zed_{2}$& 6&10& 4&no&Theorem~\ref{differencesets.thm}, $D=\{$(0,0), (0,1), (1,0), (2,0), (5,0), (6,1)$\}$\\\hline
17&$\zed_{17}$& 8& 8& 4&yes&Theorem~\ref{PDStoNF.thm}, $D=\{$1, 2, 4, 8, 9, 13, 15, 16$\}$\\\hline
19&$\zed_{19}$& 9&10& 5&no&Theorem~\ref{differencesets.thm}, $D=\{$0, 1, 2, 3, 5, 7, 12, 13, 16$\}$\\\hline
21&$\zed_{21}$& 4&10& 2&yes&Theorem~\ref{general.thm} ($\zed_{21} \cong \zed_{7} \times \zed_3$)\\\hline
21&$\zed_{21}$& 5&16& 4&no&Theorem~\ref{differencesets.thm}, $D=\{$0, 1, 4, 14, 16$\}$\\\hline
21&$\zed_{21}$& 8&10& 4&no& Example \ref{ex21} 
\\\hline
23&$\zed_{23}$&11&12& 6&no&Theorem~\ref{differencesets.thm}, $D=\{$0, 1, 2, 3, 5, 7, 8, 11, 12, 15, 17$\}$\\\hline
25&$(\zed_{5})^{2}$& 4&12& 2&yes&Theorem~\ref{general.thm}\\\hline
25&$(\zed_{5})^{2}$&12&12& 6&yes&Theorem~\ref{PDStoNF.thm}, $D=\{$(0,1), (0,2), (0,3), (0,4), (1,0), (1,1), (2,0), (2,2), (3,0), (3,3), (4,0), (4,4)$\}$\\\hline
27&$(\zed_{3})^{3}$& 8&13& 4&yes& Theorem \ref{general.thm}
\\\hline
27&$(\zed_{3})^{3}$&13&14& 7&no&Theorem~\ref{differencesets.thm}, $D=\{$(0,0,0), (0,0,1), (0,0,2), (0,1,0), (0,1,1), (0,2,0), (1,0,0), (1,0,1), (1,1,0), (2,0,1), (2,1,2), (2,2,0), (2,2,2)$\}$\\\hline
27&$\zed_{9}{\times}\zed_{3}$& 4&13& 2&yes&Theorem~\ref{general.thm}\\\hline
28&$\zed_{14}{\times}\zed_{2}$& 9&12& 4&no&$S=\{$(0,0), (0,1), (1,0), (2,0), (3,1), (4,1), (7,1), (12,0), (13,0)$\}$, $T=\{$(1,1), (3,0), (3,1), (4,1), (5,0), (6,1), (8,1), (9,0), (9,1), (11,0), (12,1), (13,1)$\}$\\\hline
29&$\zed_{29}$&14&14& 7&yes&Theorem~\ref{PDStoNF.thm}, $D=\{$1, 4, 5, 6, 7, 9, 13, 16, 20, 22, 23, 24, 25, 28$\}$\\\hline
31&$\zed_{31}$& 6&20& 4&no&$S=\{$0, 1, 2, 4, 8, 16$\}$, $T=\{$1, 2, 3, 4, 6, 7, 8, 11, 12, 13, 14, 16, 17, 19, 21, 22, 24, 25, 26, 28$\}$\\\hline
31&$\zed_{31}$& 6&25& 5&no&Theorem~\ref{differencesets.thm}, $D=\{$0, 1, 3, 8, 12, 18$\}$\\\hline
31&$\zed_{31}$&15&16& 8&no&Theorem~\ref{differencesets.thm}, $D=\{$0, 1, 2, 3, 5, 6, 8, 9, 13, 16, 21, 22, 23, 25, 27$\}$\\\hline
33&$\zed_{33}$& 4&16& 2&yes&Theorem~\ref{general.thm} ($\zed_{33} \cong \zed_{11} \times \zed_3$)\\\hline
33&$\zed_{33}$&12&16& 6&no& Example \ref{ex33} 
\\\hline
35&$\zed_{35}$& 4&17& 2&yes&Theorem~\ref{general.thm} ($\zed_{35} \cong \zed_{7} \times \zed_5$)\\\hline
35&$\zed_{35}$& 8&17& 4&no& Example \ref{ex35}
\end{longtable}
\renewcommand{\arraystretch}{1}
\end{center}  
\begin{center}
\renewcommand{\arraystretch}{1.2}
\begin{longtable}{@{}r|l|r|r|r|@{\hspace{2pt}}R{10.25cm}l@{}}
\caption{ When a $\lambda$-fold \textbf{symmetric} $(s,t)$-near-factorization of an abelian group of order $n$    exists with $\lambda \geq 2$, where $36\leq n \leq 50$,
either
an example is presented or an authority is given.
Parameters  do not appear in the table if
a symmetric near-factorization with those parameters does not exist.
(It is possible that a non-symmetric near factorization may exist for the missing parameters.)
}\\ \hline
\label{Pure.tab}
$n$ & group & $s$ & $t$ & $\lambda$  & Authority  \\ \hline       
\endfirsthead
\caption[]{(continued)}\\ \hline
$n$ & group & $s$ & $t$ & $\lambda$  & Authority  \\ \hline
\endhead
\hline
\endfoot
36&$(\zed_{6})^{2}$&15&21& 9&Theorem~\ref{differencesets.thm}, $D=\{$(0,0), (0,1), (0,3), (0,5), (1,0), (1,1), (1,3), (2,2), (3,0), (3,1), (3,5), (4,4), (5,0), (5,3), (5,5)$\}$\\\hline
37&$\zed_{37}$&18&18& 9&Theorem~\ref{PDStoNF.thm}, $D=\{$1, 3, 4, 7, 9, 10, 11, 12, 16, 21, 25, 26, 27, 28, 30, 33, 34, 36$\}$\\\hline
39&$\zed_{39}$& 4&19& 2&Theorem~\ref{general.thm} ($\zed_{39} \cong \zed_{13} \times 3$)\\\hline
39&$\zed_{39}$&12&19& 6& Example \ref{ex39}
\\\hline
41&$\zed_{41}$&20&20&10&Theorem~\ref{PDStoNF.thm}, $D=\{$1, 2, 4, 5, 8, 9, 10, 16, 18, 20, 21, 23, 25, 31, 32, 33, 36, 37, 39, 40$\}$\\\hline
45&$\zed_{15}{\times}\zed_{3}$& 4&22& 2&Theorem~\ref{general.thm}\\\hline
45&$\zed_{15}{\times}\zed_{3}$& 8&22& 4& Theorem~\ref{general.thm} ($\zed_{15}{\times}\zed_{3} \cong \zed_{5} \times \zed_3 \times \zed_3$)
\\\hline
45&$\zed_{45}$& 4&22& 2&Theorem~\ref{general.thm} ($\zed_{45} \cong \zed_{9} \times \zed_5$)\\\hline
49&$(\zed_{7})^{2}$& 4&24& 2&Theorem~\ref{general.thm}\\\hline
49&$(\zed_{7})^{2}$&24&24&12&Theorem~\ref{PDStoNF.thm}, $D=\{$(0,1), (0,2), (0,3), (0,4), (0,5), (0,6), (1,0), (1,1), (1,2), (2,0), (2,2), (2,4), (3,0), (3,3), (3,6), (4,0), (4,1), (4,4), (5,0), (5,3), (5,5), (6,0), (6,5), (6,6)$\}$
\end{longtable}
\renewcommand{\arraystretch}{1}
\end{center}

It is interesting to note that all of the half-set near-factorizations in Tables \ref{Mixed.tab} and \ref{Pure.tab} can be constructed using the product constructions, with the exception of the \NF[(24,24)]{(\zed_{7})^{2}}{12}.

In the Appendix, we provide a list of all parameters for which near-factorizations in abelian groups do not exist, but for which nonexistence only follows from exhaustive computer search, for all abelian groups of order $n \leq 35$.

\section{Summary}

We have initiated the study of $\lambda$-fold near-factorizations with $\lambda>1$. It is striking that no nontrivial $1$-fold near-factorizations of noncyclic abelian groups are known, while we have found many examples for $\lambda > 1$ in this paper. Discovering additional constructions and necessary conditions for existence of $\lambda$-fold near-factorizations is an interesting topic for future research. 

The construction of $\lambda$-fold near-factorizations of nonabelian groups has not been addressed in this paper. For $\lambda=1$, such near-factorizations commonly exist in dihedral groups, suggesting that dihedral groups might be a good place to look for $\lambda$-fold near-factorizations with $\lambda > 1$. Moreover, even for abelian groups, the nonexistence of $\lambda$-fold near-factorizations with certain parameters warrants further investigation.

\section*{Acknowledgements}
The authors would like to thank Marco Buratti for helpful comments.

\appendix

\section*{Appendix}

We provide a table of parameters for which near-factorizations in abelian groups do not exist, but whose nonexistence only follows from exhaustive computer search, for all abelian groups of order $n \leq 35$. This list is generated by the following procedure:

\begin{enumerate}
\item Apply Theorems \ref{sum.thm}, \ref{upperbound.lem}, \ref{n1n2.thm} and \ref{t2}, and Lemmas 
\ref{11.lem} and  \ref{n-1prime.lem},  to obtain the  all quadruples $(G,s,t,\lambda)$ that satisfy the known necessary conditions.
\item  Remove quadruples with $|G|=s+t$ and $|G|=s+t+1$, as the existence or nonexistence of the associated difference sets and partial difference sets is known. This  leaves $144$  quadruples $(G,s,t,\lambda)$.
\item Remove the quadruples $ (G,s,t,\lambda)$ that were found by computer search and presented in Table 1 (i.e., those that are not derived from difference sets or partial difference sets). There are 13 such quadruples. The remaining $144-13=131$ quadruples $ (G,s,t,\lambda)$ are listed in Table 3.
\end{enumerate}


\renewcommand{\arraystretch}{1.2}
\begin{table}[H]
\caption{Parameters for which near-factorizations in abelian groups do not exist, but whose nonexistence only follows from exhaustive computer search, for all abelian groups of order $n \leq 35$.
}\label{Computer.tab}
\begin{center}
\begin{tabular}{l@{\qquad}l}
\begin{tabular}{@{}r|R{3cm}|r|r|r@{\hspace{2pt}}l@{}}\hline
$n$ & group  $G$ & $s$ & $t$ & $\lambda$    \\ \hline       
11 & $\zed_{11}$& 4& 5& 2\\
13 & $\zed_{13}$& 3& 8& 2 \\
13& $\zed_{13}$& 4& 6& 2\\
15& $\zed_{15}$& 6& 7& 3\\
16& $(\zed_2)^4$, $\zed_4 \times (\zed_2)^2$, $(\zed_4)^2$, $\zed_8 \times \zed_2$, $\zed_{16}$ & 3& 10& 2\\[3.5ex]
16&  $(\zed_2)^4$, $\zed_4 \times (\zed_2)^2$, $(\zed_4)^2$, $\zed_8 \times \zed_2$, $\zed_{16}$ & 5& 6& 2\\[3.5ex]
16&  $(\zed_2)^4$, $\zed_4 \times (\zed_2)^2$, $(\zed_4)^2$,  $\zed_{16}$  & 5& 9& 3\\[3.5ex]
17& $\zed_{17}$& 4& 8& 2\\
17& $\zed_{17}$& 6& 8& 3\\
19& $\zed_{19}$& 3& 12& 2 \\
19& $\zed_{19}$& 4& 9& 2\\
19& $\zed_{19}$& 6& 6& 2\\
19& $\zed_{19}$& 6& 9& 3\\
19& $\zed_{19}$& 8& 9& 4 \\
21 & $\zed_{21}$& 5& 8& 2\\
21 & $\zed_{21}$& 4& 15& 3\\
21 & $\zed_{21}$& 5& 12& 3\\
21 & $\zed_{21}$& 6& 10& 3\\
22 & $\zed_{22}$& 3& 14& 2\\
22 & $\zed_{22}$& 6& 7& 2\\
22 & $\zed_{22}$& 7& 9& 3\\
22 & $\zed_{22}$& 6& 14& 4\\
22 & $\zed_{22}$& 7& 12& 4\\
23 & $\zed_{23}$& 4& 11& 2\\
23 & $\zed_{23}$& 6& 11& 3\\
23 & $\zed_{23}$& 8& 11& 4\\
23 & $\zed_{23}$& 10& 11& 5\\
25 & $(\zed_5)^2$, $\zed_{25}$& 3& 16& 2\\
25 & $\zed_{25}$                   & 4& 12& 2\\
25 & $(\zed_5)^2$, $\zed_{25}$& 6& 8& 2\\
25 & $(\zed_5)^2$, $\zed_{25}$& 4& 18& 3\\
25 & $(\zed_5)^2$, $\zed_{25}$& 6& 12& 3\\
25 & $(\zed_5)^2$, $\zed_{25}$& 8& 9& 3\\
25 & $(\zed_5)^2$, $\zed_{25}$& 6& 16& 4\\
\hline \end{tabular}
&
\begin{tabular}{@{}r|R{3cm}|r|r|r@{\hspace{2pt}}l@{}}\hline
$n$ & group  $G$ & $s$ & $t$ & $\lambda$    \\ \hline
25 & $(\zed_5)^2$, $\zed_{25}$& 8& 12& 4\\
25 & $(\zed_5)^2$, $\zed_{25}$& 8& 15& 5\\
25 & $(\zed_5)^2$, $\zed_{25}$& 10& 12& 5\\
26 & $\zed_{26}$& 5& 10& 2\\
26 & $\zed_{26}$& 5& 15& 3\\
26 & $\zed_{26}$& 10& 10& 4\\
27 & $(\zed_3)^3$, $\zed_{27}$                                  & 4& 13& 2\\
27 & $(\zed_3)^3$, $\zed_9 \times\zed_3$, $\zed_{27}$& 6& 13& 3\\
27 & $ \zed_9 \times\zed_3$, $\zed_{27}$& 8& 13& 4\\
27 & $(\zed_3)^3$, $\zed_9 \times\zed_3$, $\zed_{27}$& 10& 13& 5\\
27 & $(\zed_3)^3$, $\zed_9 \times\zed_3$, $\zed_{27}$& 12& 13& 6\\
28 & $\zed_{14} \times\zed_2$, $\zed_{28}$& 3& 18& 2\\
28 & $\zed_{14} \times\zed_2$, $\zed_{28}$& 6& 9& 2\\
28 & $\zed_{14} \times\zed_2$, $\zed_{28}$& 9& 9& 3\\
28 & $\zed_{14} \times\zed_2$, $\zed_{28}$& 6& 18& 4\\
28 & $\zed_{28}$                                       & 9& 12& 4\\
28 & $\zed_{14} \times\zed_2$, $\zed_{28}$& 9& 15& 5\\
29 & $\zed_{29}$& 4& 14& 2\\
29 & $\zed_{29}$& 7& 8& 2\\
29 & $\zed_{29}$& 4& 21& 3\\
29 & $\zed_{29}$& 6& 14& 3\\
29 & $\zed_{29}$& 7& 12& 3\\
29 & $\zed_{29}$& 7& 16& 4\\
29 & $\zed_{29}$& 8& 14& 4\\
29 & $\zed_{29}$& 7& 20& 5\\
29 & $\zed_{29}$& 10& 14& 5\\
29 & $\zed_{29}$& 12& 14& 6\\
31 & $\zed_{31}$& 3& 20& 2\\
31 & $\zed_{31}$& 4& 15& 2\\
31 & $\zed_{31}$& 5& 12& 2\\
31 & $\zed_{31}$& 6& 10& 2\\
31 & $\zed_{31}$& 5& 18& 3\\
31 & $\zed_{31}$& 6& 15& 3\\
31 & $\zed_{31}$& 9& 10& 3\\
31 & $\zed_{31}$& 5& 24& 4\\
31 & $\zed_{31}$& 8& 15& 4\\
31 & $\zed_{31}$& 10& 12& 4\\
\hline\end{tabular}
\end{tabular}
\end{center}
\end{table}

\begin{table}[H]
\caption*{Table \ref{Computer.tab}: (continued) }
\begin{center}
\begin{tabular}{l@{\qquad}l}
\begin{tabular}{@{}r|R{3cm}|r|r|r@{\hspace{2pt}}l@{}}\hline
$n$ & group  $G$ & $s$ & $t$ & $\lambda$    \\ \hline
31 & $\zed_{31}$& 10& 15& 5\\
31 & $\zed_{31}$& 9& 20& 6\\
31 & $\zed_{31}$& 10& 18& 6\\
31 & $\zed_{31}$& 12& 15& 6\\ 
31 & $\zed_{31}$& 14& 15& 7\\ 
33 & $\zed_{33}$& 8& 8& 2\\ 
33 & $\zed_{33}$& 4& 24& 3\\ 
33 & $\zed_{33}$& 6& 16& 3\\ 
33 & $\zed_{33}$& 8& 12& 3\\ 
33 & $\zed_{33}$& 8& 16& 4\\ 
33 & $\zed_{33}$& 8& 20& 5\\ 
33 & $\zed_{33}$& 10& 16& 5\\ 
33 & $\zed_{33}$& 14& 16& 7\\ 
34 & $\zed_{34}$& 3& 22& 2\\
\hline\end{tabular}
&
\begin{tabular}{@{}r|R{3cm}|r|r|r@{\hspace{2pt}}l@{}}\hline
$n$ & group  $G$ & $s$ & $t$ & $\lambda$    \\ \hline
34 & $\zed_{34}$& 6& 11& 2\\ 
34 & $\zed_{34}$& 9& 11& 3\\ 
34 & $\zed_{34}$& 6& 22& 4\\ 
34 & $\zed_{34}$& 11& 12& 4\\ 
34 & $\zed_{34}$& 11& 15& 5\\ 
34 & $\zed_{34}$& 9& 22& 6\\ 
34 & $\zed_{34}$& 11& 18& 6\\ 
34 & $\zed_{34}$& 11& 21& 7 \\
35 & $\zed_{35}$& 6& 17& 3\\ 
35 & $\zed_{35}$& 10& 17& 5\\ 
35 & $\zed_{35}$& 12& 17& 6\\ 
35 & $\zed_{35}$& 14& 17& 7\\ 
35 & $\zed_{35}$& 16& 17& 8\\ 
\hline
\multicolumn{5}{c}{}
\end{tabular}
\end{tabular}
\end{center}
\end{table}

\renewcommand{\arraystretch}{1}

\end{document}